\UseRawInputEncoding
\documentclass[11pt]{amsart}
\usepackage{amsxtra,amssymb,amsmath,amscd,url,listings,stmaryrd}
\usepackage[alphabetic,initials]{amsrefs}
\usepackage{cases}
\usepackage{mathrsfs}
\usepackage{bbm}
\usepackage{amssymb}
\usepackage{amscd}
\usepackage{amsfonts,latexsym,amsthm,amsxtra,mathdots,amssymb,latexsym}
\usepackage[all,cmtip]{xy}
\RequirePackage{amsmath} \RequirePackage{amssymb}
\usepackage{color}
\usepackage{colordvi}
\usepackage{multicol}
\usepackage{hyperref}
\usepackage{mathtools}
\usepackage[margin=1in]{geometry}
\usepackage{xcolor}
\hypersetup{
    colorlinks,
    linkcolor={red!50!black},
    citecolor={blue!50!black},
    urlcolor={blue!80!black}
}

\newtheorem{conjecture}{Conjecture}[section]
\newtheorem{problem}{Problem}[section]

\newtheorem{theorem}{Theorem}[section]

\newtheorem*{proposition*}{Proposition}
\newtheorem{lemma}{Lemma}

\newtheorem{Remark}{Remark}[section]
\newtheorem{Principle}{Principle}[section]

\newcommand{\y}{Y}

\newcommand{\M}{{\mathscr{M}}}
\newcommand{\B}{{\mathscr{B}}}

\newcommand{\NN}{{\mathscr{N}}}

\newcommand{\N}{{\mathbb N}}

\newcommand{\lb}{\left(}
\newcommand{\rb}{\right)}
\def \i {{\rm i}}
\def \d {{\rm d}}

\newcommand{\J}{{\mathscr{U}}}

\newcommand{\R}{{\mathbb R}}

\newcommand{\Er}{{\mathscr{E}}}

\begin{document}
\title{ Omega theorems for logarithmic derivatives  of zeta and $L$-functions}
 \author[Daodao Yang]{Daodao Yang}

\address{Address: Universit\'e de Montr\'eal,
CP 6128, Succ. Centre-Ville
H3C 3J7, Montr\'eal, QC, Canada}

\email{yangdao2@126.com \quad dyang@msri.org}

\maketitle

\begin{abstract}
We establish several new   $\Omega$-theorems for logarithmic derivatives  of the Riemann zeta function and  Dirichlet $L$-functions. In particular,   this improves on earlier work   of Landau (1911), Bohr-Landau (1913), and recent work of  Lamzouri.

\end{abstract}

 \section{Introduction}

The  logarithmic derivative of the Riemann zeta function is an important object in the analytic number theory since it naturally appears in the proof of the prime number theorem. In 1911, Landau \cite{Lan1911} proved that there exists a positive constant $k_1$, such that the following inequality 
\[ \left| \frac{\zeta^{\prime}}{\zeta}(s)\right| \geqslant k_1 \log_2 t, 
\]
has a solution $s = \sigma + \i t$ in the region $\{\sigma + \i t: \sigma > 1, \quad  t > \tau\}$ for any given $\tau  > 0$. From Landau's $\Omega$-theorem, one can deduce that $ \zeta^{\prime}(1 + \i t)/\zeta(1 + \i t) = \Omega \lb \log_2 t \rb$ by  a simple application of   the Phragm\'en--Lindel\"of principle.     And in 1913, Bohr-Landau \cite{BL1913} proved that there exists  a positive constant $k_2$, such that for any given  $\theta$ with $0 < \theta < \frac{1}{2}$ and for any given   $\tau > 0$ the following inequality 
\[ \left| \frac{\zeta^{\prime}}{\zeta}(s)\right| \geqslant \lb \log t \rb^{k_2 \theta}, 
\]
has a solution $s = \sigma + \i t$ in the region $\{\sigma + \i t: \sigma \geqslant 1 - \theta, \quad  t \geqslant \tau\}$. However, both  $k_1$ and $k_2$ in their results are not effective constants. In this paper, we will show that one can take $k_1 = k_2 = 1$. 

After the work of Bohr and Landau, when assuming the  the Riemann hypothesis(RH), Littlewood \cite{Lit1925} proved that the following upper bound (also see \cite{MV}*{Theorem 13.13}),     

\[
\left| \frac{\zeta^{\prime}}{\zeta}(\sigma + \i t)\right| \leqslant \sum_{ n \leqslant (\log t)^2} \frac{\Lambda(n)}{n^{\sigma}}+ O\left( \lb\log t\rb^{2 - 2\sigma} \right),\]
holds uniformly for $~~\frac{1}{2} + \frac{1}{\log_2 t} \leqslant \sigma \leqslant \frac{3}{2},$ and $t \geqslant 10.$ In particular, the following holds on RH
 \begin{align*}
  \left| \frac{\zeta^{\prime}}{\zeta}\left(1 + \i t\right)\right|  \leqslant 2 \log_2 t + O(1),\, \quad \text{and} \quad \left| \frac{\zeta^{\prime}}{\zeta}(\sigma + \i t)\right| \ll_{\epsilon} \lb \log t\rb^{2 - 2\sigma}, \quad \forall \frac{1}{2} 
 + \epsilon \leqslant \sigma   \leqslant 1 -\epsilon,\quad \forall t \geqslant 10,
\end{align*}
for any fixed positive number $\epsilon < \frac{1}{10}.$

When assume the Generalized Riemann Hypothesis(GRH), one can establish similar upper bounds  for Dirichlet $L$-functions.   In the past years, there were many studies on upper bounds of logarithmic derivatives  of zeta and $L$-functions. For instance, see \cites{ IMS,  Tru, CG22,  EP, PS, C23, AS23}. But the orders of Littlewood's conditional upper bounds have never been improved. Only the big $O(1)$ term for the $1$-line and the implied constant for the $\sigma$-line when $\sigma \in (\frac{1}{2}, 1)$ have been improved.

 An important motivation for the research on  values of logarithmic derivatives of Dirichlet $L$-functions is that values of $ L^{\prime}(1, \chi)/L(1, \chi)$ are related to  the Euler-Kronecker constants  of the cyclotomic fields. Systematic study of the Euler-Kronecker constants was started by Ihara \cites{Iha1, Iha2}. Let $K$ be  a global field, and $\zeta_K(s)$  be the associated Dedekind zeta function. Then the  Euler-Kronecker constant $\gamma_K$ is defined by
\[\gamma_K:\, = \lim_{s\to 1} \lb\frac{\zeta^{\prime}_K(s)}{\zeta_K(s)} + \frac{1}{s-1}\rb.\]When $K$ is the cyclotomic field $\mathbb{Q}(\xi_p)$ generated by a   primitive $p$-th root of unity,   one may deduce
\[\gamma_K = \gamma + \sum_{\chi \neq \chi_0}\frac{L^{\prime}}{L} (1, \chi).\]

In 2009, assuming GRH,  Ihara-Kumar Murty-Mahoro Shimura \cite{IMS} established\footnote{This is Corollary 3.3.2 in  \cite{IMS} and the constant in the big $O(\cdot)$ term is stated explicitly.} that
\[ \left|\frac{L^{\prime}}{L} (1, \chi) \right| \leqslant 2 \log_2 q + 2(1-\log 2) + O\lb \frac{\log_2 q}{\log q} \rb,\]
for all nonprincipal primitive characters $\chi$(mod $q$).    In 2023, assuming GRH, Chirre-Val\r{a}s Hagen-Simoni\u{c} \cite{C23} made further improvement on the constant term, and obtained\footnote{Note that $- 0.4989 < 2(1-\log 2) = 0.613\cdots.$}
 \[ \left|\frac{L^{\prime}}{L} (1, \chi) \right|\leqslant 2 \log_2 q -0.4989 +  5.91 \frac{\lb \log_2 q \rb^2}{\log q} , \quad \forall q \geqslant 10^{30},\]
for all  primitive characters $\chi$(mod $q$) and they also  obtained similar result for the  the Riemann zeta function
 \[ \left|\frac{\zeta^{\prime}}{\zeta}(1 + \i t) \right| \leqslant 2 \log_2 t -0.4989 + 5.35 \frac{\lb \log_2 t \rb^2}{\log t} , \quad \forall t \geqslant 10^{30}, \]
when assume RH.

 Compared to research on upper bounds, there are  fewer results for  lower bounds of maximal values of logarithmic derivatives  of zeta and $L$-functions.   Let $D$ be a fundamental discriminant and let $\chi_D$ be the quadratic Dirichlet character of conductor $D$.    In \cite{MKM}, M. Mourtada-Kumar Murty proved that  there are infinitely many fundamental discriminants $D$ (both positive
and negative) such that $L^{\prime}(1, \chi_D)/L (1, \chi_D) \geqslant  \log_2 |D| +O(1)$. Furthermore, for $x$ large enough, they proved that there are $\geqslant x^{\frac{1}{10}}$ fundamental discriminants $ 0 <
D \leqslant x$ such that $-L^{\prime}(1, \chi_D)/L (1, \chi_D) \geqslant  \log_2 D +O(1)$. When assume GRH, M. Mourtada-Kumar Murty proved that for $x$ large enough, there are $\gg x^{\frac{1}{2}}$ primes $q \leqslant x$ such
that $L^{\prime}(1, \chi_q)/L (1, \chi_q) \geqslant  \log_2 x + \log_3 x + O(1)$ and $\gg x^{\frac{1}{2}}$ primes $q \leqslant x$ such
that $-L^{\prime}(1, \chi_q)/L (1, \chi_q) \geqslant  \log_2 x + \log_3 x + O(1)$.

And in \cites{Paul}, Paul established $\Omega$-theorems for $L^{\prime}(1, f, \chi)/L(1, f, \chi)$, where  $L(s, f, \chi)$ is the $L$-function associated by a holomorphic cusp form $f$ and a quadratic Dirichlet   character $\chi$. But the leading coefficient in Paul's result is $\frac{1}{8}$ instead of $1$. 

Furthermore, one can consider the family of Dirichlet $L$-functions associated by all  nonprincipal characters mod $q$, not just the quadratic characters. It was stated in \cite{LARL} that Lamzouri remarked that, by adapting the techniques in his paper \cite{Lamzouri}, one can establish that  $M_q \geqslant (1 + o(1)) \log_2 q$, as prime $q \to \infty$, where $M_q$ is defined as follows
\begin{align*}
    M_q: =   \max_{ \substack{  \chi \neq \chi_0 \\ \chi\,\lb\text{mod}\, q\rb}}\left|\frac{L^{\prime}}{L}(1, \chi)\right|.
\end{align*}
 For numerical computations for $M_q$, see \cites{LA, LARL}. And we  mention that in \cite{LYLA}, small values of $ |L^{\prime}(1, \chi)/L(1, \chi)|$ were studied.

In this paper, we will show that $M_q$ can be even larger.  Moreover, we  establish $\Omega$-results for any $\sigma \in (\frac{1}{2}, 1]$, not just limited to the case $\sigma = 1$. These results are established for  the Riemann zeta function as well, which improves on the work of Bohr and Landau  \cites{Lan1911, BL1913}.

First, we have the following result for  the logarithmic derivative of the Riemann zeta function on the $1$-line.
 \begin{theorem}\label{Zetasig=1}
 Let $\beta, \epsilon \in (0, 1)$ be fixed. Then for all sufficiently large $T$, we have
 \begin{align*}
\max_{T^{\beta} \leqslant t \leqslant  T}-\emph{Re} \frac{\zeta^{\prime}}{\zeta}(1 + \i t) \geqslant \log_2 T + \log_3 T + C_1 -\epsilon,
\end{align*}
 where $C_1$ is an explicit constant defined by \eqref{C1}. In particular, for all sufficiently large $T$, we have
 \begin{align*}
\max_{T^{\beta} \leqslant t \leqslant  T}\left|\frac{\zeta^{\prime}}{\zeta}(1 + \i t) \right|\geqslant \log_2 T + \log_3 T + C_1 -\epsilon.
\end{align*}
 \end{theorem}

\begin{Remark}
    By the continuity of $\zeta^{\prime}(s)/\zeta(s)$, we find that 
    the following inequality 
\[ \left| \frac{\zeta^{\prime}}{\zeta}(s)\right| \geqslant  \log_2 t + \log_3 t+ C_1 -\epsilon, 
\]
has a solution $s = \sigma + \i t$ in the region $\{\sigma + \i t: \sigma > 1, \quad  t > \tau\}$ for any given $\tau  > 0$. 
\end{Remark}

In the next theorem, we give a lower bound for the measure of the set of $t$ for which  the  logarithmic derivative of the Riemann zeta function is large.

\begin{theorem}\label{Zetasig=1: meas}
 Let $\beta \in (0, 1)$ be fixed. Let $x >0$ be given, and define the function $\M(T, x)$ as
 \begin{align*}
\M(T, x) = \emph{meas}\{t \in [T^{\beta},  T]: -\emph{Re} \frac{\zeta^{\prime}}{\zeta}(1 + \i t) \geqslant \log_2 T + \log_3 T + C_1 -x\},
\end{align*}
 where $C_1$  is the constant defined by  \eqref{C1}. Then we have
 \[\M(T, x)\geqslant T^{1-\lb1-\beta\rb e^{-x} + o\lb1\rb},\quad as\quad T \to \infty.\]
 \end{theorem}

 The following two theorems for Dirichlet $L$-functions are similar to the above two theorems. 
 
\begin{theorem}\label{Lsig=1}
 Let $\epsilon \in (0, 1)$ be fixed. Then for all sufficiently large prime $q$, we have
 \begin{align*}
 \max_{ \substack{  \chi \neq \chi_0 \\ \chi\,\lb\emph{mod}\, q\rb}}-\emph{Re} \frac{L^{\prime}}{L}(1, \chi) \geqslant \log_2 q + \log_3 q + C_2 - \epsilon,
\end{align*}
 where $C_2$ is an explicit constant defined by \eqref{C2}.  In particular, for all sufficiently large prime $q$, we have
 \begin{align*}
      \max_{ \substack{  \chi \neq \chi_0 \\ \chi\,\lb\emph{mod}\, q\rb}}\left|\frac{L^{\prime}}{L}(1, \chi)\right| \geqslant \log_2 q + \log_3 q + C_2 - \epsilon.
 \end{align*}
 \end{theorem}

\begin{theorem}\label{Lsig=1: number}
Let $x >0$ be given, and define the counting function $\NN(q, x)$ as
 \begin{align*}
\NN(q, x) = \#\{\chi\,\lb\emph{mod}\, q\rb : -\emph{Re} \frac{L^{\prime}}{L}(1, \chi) \geqslant \log_2 q + \log_3 q + C_2 - x\},
\end{align*}
 where $C_2$  is the constant defined by  \eqref{C2}. For prime $q \to \infty$, we have
 \[\NN(q, x)\geqslant q^{1-e^{-x} + o\lb1\rb}.\]
 \end{theorem}

The next two theorems are concerning the large values inside the critical strip.

\begin{theorem}\label{Zeta: strip}
 Let $\beta \in (0, 1)$, $\sigma \in (\frac{1}{2}, 1)$  be fixed. Then for all sufficiently large $T$, we have
 \begin{align*}
\max_{T^{\beta} \leqslant t \leqslant  T}-\emph{Re} \frac{\zeta^{\prime}}{\zeta}(\sigma + \i t)\geqslant  C_3\lb\sigma\rb   (\log T)^{1-\sigma}(\log_2 T)^{1-\sigma},
\end{align*}
where $C_3(\sigma)$ is some positive constant   which can be effectively computed.
 \end{theorem}

\begin{theorem}\label{L: strip}
 Let $\sigma \in (\frac{1}{2}, 1)$  be fixed. Then for all sufficiently large prime $q$, we have
 \begin{align*}
 \max_{ \substack{  \chi \neq \chi_0 \\ \chi\,\lb\emph{mod}\, q\rb}} -\emph{Re} \frac{L^{\prime}}{L}(\sigma, \chi)\geqslant  C_4\lb\sigma\rb  (\log q)^{1-\sigma}(\log_2 q)^{1-\sigma},
\end{align*}
where $C_4(\sigma)$ is some positive constant   which can be effectively computed.
 \end{theorem}

\begin{Remark}
    The methods used in the proof of Theorem \ref{Zetasig=1: meas} and \ref{Lsig=1: number}  could be used to estimate the measure or the cardinality in the above two theorems. However, these kinds of results would be  weaker than those in Theorem \ref{Zetasig=1: meas} and \ref{Lsig=1: number}. So we have chosen not to pursue this avenue to maintain brevity in the paper.
\end{Remark}

We could  establish $\Omega$-theorems for $\text{Re}\lb e^{-\i \theta}\zeta^{\prime}(\sigma + \i t)/\zeta(\sigma + \i t) \rb$ and $\text{Re}\lb e^{-\i \theta}L^{\prime}(\sigma, \chi)/L(\sigma, \chi) \rb$ as well, but the results are weaker than the results for $-\text{Re}\lb \zeta^{\prime}(\sigma + \i t)/\zeta(\sigma + \i t) \rb$ and $-\text{Re}\lb L^{\prime}(\sigma, \chi)/L(\sigma, \chi) \rb$.  The following are two examples for $\sigma = 1$. 

\begin{theorem}\label{Zeta=1: theta}
 Let $\beta \in (0, 1)$  be fixed.   For $T\geqslant 10^{10}$,  we have \begin{align*}
\max_{T^{\beta}  \leqslant | t | \leqslant  T}\emph{Re} \lb e^{-\i \theta }\frac{\zeta^{\prime}}{\zeta}\lb1 + \i t\rb \rb \geqslant \log_2 T + O_{\beta}\lb\log_3 T \rb ,
\end{align*}
uniformly for all  $\theta \in [0, 2\pi]$. 
\end{theorem}

\begin{theorem}\label{Lsig=1: theta}
For all  prime $q\geqslant 10^{10}$, we have
 \begin{align*}
 \max_{ \substack{  \chi \neq \chi_0 \\ \chi\,\lb\emph{mod}\, q\rb}} \emph{Re} \lb e^{-\i \theta} \frac{L^{\prime}}{L}(1, \chi) \rb \geqslant \log_2 q + O\lb\log_3 q\rb,
\end{align*}
 
uniformly for all  $\theta \in [0, 2\pi]$.
 \end{theorem}

\begin{Remark}
 When $\sigma \in (\frac{1}{2}, 1)$  is fixed,  we mention that one can use Montgomery's  method \cite{MZeta} to establish the $\Omega$-result that $\text{Re}\lb e^{-\i \theta}\zeta^{\prime}(\sigma + \i t)/\zeta(\sigma + \i t) \rb $ $=       \Omega_{+}\lb  (\log t)^{1-\sigma}(\log_2 t)^{1-\sigma}\rb$,   $\forall \theta \in [0, 2\pi]$, but the implied constant is smaller than the constant in Theorem \ref{Zeta: strip}. It is unclear how to  establish the similar result that $\text{Re}\lb e^{-\i \theta}L^{\prime}(\sigma, \chi)/L(\sigma, \chi) \rb$ $=       \Omega_{+}\lb  (\log q)^{1-\sigma}(\log_2 q)^{1-\sigma}\rb$, $\forall \theta \in [0, 2\pi]$, for Dirichlet $L$-functions ($\chi \neq \chi_0$).
\end{Remark}

In the paper, we employ resonance methods to establish the theorems. Specifically, our proof for the first six theorems relies on \textit{long resonator methods}, employing methods of Aistleitner-Mahatab-Munsch \cite{AMM} and Aistleitner-Mahatab-Munsch-Peyrot \cite{AMMP}. And to prove the last two theorems, we use \textit{short resonator methods},  employing methods of Soundararajan
\cite{SoundExtreme}, and combining ideas of  Bondarenko--Seip \cite{BS}. The reason is that for such general cases we are no longer able to use the long resonator methods as in \cites{AMM, AMMP}. Resonance methods first appeared in Voronin's work \cite{Voronin}, but his work did not receive sufficient attention.  A more general and powerful resonance method was introduced by Soundararajan
\cite{SoundExtreme}. Since then, resonance methods have been widely utilized to produce extreme values for zeta and $L$-functions. Later, Aistleitner \cite{CA} pioneered long resonator methods, further developing the resonance methods. For more information on  resonance methods, we refer readers to \cites{Voronin,  SoundExtreme, Hough, CA,  BS, BS2,  AMM, AMMP, dT, Sound}.

\bigskip

\textbf{Notations:} 1) in this paper, we use the short-hand
notations, $\,\log_2 T :\,= \log\log T,$  $\log_3 T :\,= \log\log\log T$ and so on. 2) $p$ denotes a prime. 3) $A$ denotes an absolute positive constant, which  may vary from line to line. And $A_{\beta}$ denotes a positive constant, which  may depend on $\beta$ and may vary from line to line. 4) $\chi_0$ denotes the principal character. 
\bigskip


\section{Proof of Theorem  \ref{Zetasig=1}}\label{proofZetasig=1}
First, we establish an approximation formula for $\zeta^{\prime}(1+ \i t)/\zeta(1+ \i t)$.
Set $ c = \frac{1}{\log Y}$ and $Y = \exp{\lb\lb \log T\rb^2\rb}$.  We use an effective version of Perron's formula (see \cite{koukou}*{Theorem 7.2}) to obtain the following\begin{align*}
    \sum_{n\leqslant Y} \frac{\Lambda(n)}{n^{1 +\i t}} =  \frac{1}{2\pi \i} \int_{c - \i \frac{T^{\beta}}{2}}^{c + \i \frac{T^{\beta}}{2}} 
-\frac{\zeta^{\prime}( 1 + \i t +s)}{\zeta( 1 + \i t +s)} \frac{Y^s}{s} \d s  + O \lb \frac{(\log Y)^2}{T^{\beta}} + \frac{\log Y}{Y}\rb .
\end{align*}

By the residue theorem,  we get
\begin{align*}
  \frac{1}{2\pi \i} \int_{c - \i \frac{T^{\beta}}{2}}^{c + \i \frac{T^{\beta}}{2}} 
-\frac{\zeta^{\prime}( 1 + \i t +s)}{\zeta( 1 + \i t +s)} \frac{Y^s}{s} \d s  = -\frac{\zeta^{\prime}( 1 + \i t )}{\zeta( 1 + \i t )}  + \frac{1}{2 \pi \i} \lb \int_{ c - \i  \frac{T^{\beta}}{2}}^{-\delta - \i \frac{T^{\beta}}{2}} +  \int_{-\delta - \i \frac{T^{\beta}}{2}}^{-\delta + \i \frac{T^{\beta}}{2}}   +  \int_{-\delta + \i \frac{T^{\beta}}{2}}^{c + \i \frac{T^{\beta}}{2}} \rb,
\end{align*}
 where $\delta = \frac{A}{\log T}$ will be appropriately chosen to avoid zeros of $\zeta$. Now we  make use of  the classical result of the zero-free region of $\zeta$ and upper bounds for $|\zeta^{\prime}/\zeta|$ in the zero-free region (for example, see \cite{MV}*{Theorem 6.7}).  We calculate that the second integral on the right side of above formula is bounded by
 \begin{align*}
     \ll (\log T)   Y^{-\delta} \int_{-\frac{T^{\beta}}{2}}^{\frac{T^{\beta}}{2}} \frac{\d y}{\sqrt{\delta^2 + y^2}} \ll (\log T)^2\cdot\exp{\lb-A \lb\log T\rb\rb},
 \end{align*}
and  the first and third integrals on the right side of above formula are bounded by
\begin{align*}
    \ll \frac{\log T}{T^{\beta}} \frac{1}{\log Y}\lb Y^c - Y^{-\delta} \rb \ll \frac{1}{T^{\beta}\log T}. 
\end{align*}

Taking into account the above information, we arrive at the following approximation
\begin{align}\label{zeta: approx1}
    \sum_{n\leqslant Y} \frac{\Lambda(n)}{n^{1 +\i t}} = -\frac{\zeta^{\prime}( 1 + \i t )}{\zeta( 1 + \i t )}  + O\lb T^{-A_{\beta}} \rb,\quad \forall T^{\beta} \leqslant t \leqslant T.
\end{align}

 Next we take $X = B\log T \log _2 T$, with $B = \frac{1 - \beta}{\log 4}e^{-\frac{\epsilon}{2}}$. By the definition of $B$, we have   $\beta-1 + (\log 4) B < 0$.  Define $\Phi(t) = e^{ -\frac{t^2}{2} }$ as in \cite{BS}. Note that the Fourier transform  of $\Phi$ is $\widehat \Phi(\xi):\,= 
 \int_{-\infty}^{\infty} \Phi(x) e^{-\i x\xi} dx = \sqrt{2\pi} \Phi(\xi)$, which is positive for all $\xi \in \R$.

Define $M_2(R, T)$ and $M_1(R, T)$ as follows ($R(t)$ will be defined soon)
\begin{align*}
    M_2(R, T) &=\int_{T^\beta}^{T} \text{Re}\lb   \sum_{n\leqslant Y} \frac{\Lambda(n)}{n^{1 +\i t}} \rb |R(t)|^2 \Phi(\frac{t\log T}{T}) \d t,\\
    M_1(R, T) & = \int_{T^\beta}^{T}  |R(t)|^2 \Phi(\frac{t\log T}{T}) \d t,
  \end{align*}
  then \begin{align}\label{ratioM2M1}
      \max_{T^{\beta} \leqslant t \leqslant  T}\text{Re}\lb   \sum_{n\leqslant Y} \frac{\Lambda(n)}{n^{1 +\i t}} \rb \geqslant \frac{M_2(R, T)}{M_1(R, T)}.
  \end{align}

As in \cites{AMM}, let  $r(n)$ be a completely multiplicative function and define its value at primes by 
 $$
 r(p)  = \begin{cases}
	1 - \frac{p}{X}, & \text{if\, $p \leqslant X$,}\\
            0, & \text{if\, $p > X$.}
		 \end{cases}
$$

Define $R(t)$ as follows
\begin{align*}
    R(t)  = \prod_p \frac{1}{1-\frac{r(p)}{p^{\i t}}} = \sum_{n \in \N} \frac{r(n)}{n^{\i t}}.
    \end{align*}
    
 Then by the prime number theorem we obtain \begin{align*}
   \log |R(t)| \leqslant \sum_{p \leqslant X} \log |\frac{1}{1-r(p)}|\leqslant  \frac{X}{\log X} + O\lb \frac{X}{(\log X)^2}\rb.  
 \end{align*}
 
 Thus $|R(t)|^2 \leqslant T^{2 B + o(1)}$. Clearly, we have the following bound by the choice of $Y$,\begin{align}\label{UpperBoundPartialSum}
    \left|\text{Re}\lb  \sum_{n\leqslant Y} \frac{\Lambda(n)}{n^{1 +\i t}} \rb\right| \leqslant \sum_{n\leqslant Y} \frac{\Lambda(n)}{n} \ll \log Y  \ll T^{o(1)}.
\end{align}

So we immediately have
\begin{align*}
    \left|  \int_{|t| \geqslant T} \text{Re}\lb   \sum_{n\leqslant Y} \frac{\Lambda(n)}{n^{1 +\i t}} \rb |R(t)|^2 \Phi(\frac{t\log T}{T}) \d t \right|\ll e^{-\frac{1}{4}(\log T)^2} \ll 1,
\end{align*}
and
\begin{align*}
    \left|  \int_{|t| \leqslant T^{\beta}} \text{Re}\lb   \sum_{n\leqslant Y} \frac{\Lambda(n)}{n^{1 +\i t}} \rb |R(t)|^2 \Phi(\frac{t\log T}{T}) \d t \right|\ll T^{\beta + 2B + o(1)}.
\end{align*}

As a result, we find that $2M_2(R, T) = I_2(R, T) + O(T^{\beta + 2B +o(1)})$, where $ I_2(R, T)$ is defined as\begin{align*}
     I_2(R, T) &=\int_{-\infty}^{+\infty} \text{Re}\lb   \sum_{n\leqslant Y} \frac{\Lambda(n)}{n^{1 +\i t}} \rb |R(t)|^2 \Phi(\frac{t\log T}{T}) \d t.
\end{align*}

And clearly, we have $2M_1(R, T) \leqslant I_1(R, T)$ and similarly $2M_1(R, T) = I_1(R, T) + O(T^{\beta + 2B +o(1)})$, where $ I_1(R, T)$ is defined as
\begin{align*}
    I_1(R, T) & = \int_{-\infty}^{+\infty}  |R(t)|^2 \Phi(\frac{t\log T}{T}) \d t.
\end{align*}

To give a lower bound for $I_1(R, T)$, we expand the product of $R(t)$, and in the summation, we only consider the contributions from the diagonal terms  (we can drop other terms since $\widehat \Phi(\xi)$ is positive):
\begin{align*}
    I_1(R, T) & = \int_{-\infty}^{+\infty} \sum_{\substack{k, m \in \N}} r(k)r(m) \lb\frac{k}{m}\rb^{\i t}\Phi(\frac{t\log T}{T}) \d t \\
    &\geqslant \sqrt{2\pi}\frac{T}{\log T}\sum_{n\in N} r^2(n)\\ &\geqslant T^{1+B(2-\log 4) + o(1)},
\end{align*}
where the last step follows from partial summation and the prime number theorem (see \cite{AMMP}*{page 841}). By the above estimate for $I_1(R, T)$,  we obtain
\begin{align}\label{M2M1I2I1}
    \frac{M_2(R, T)}{M_1(R, T)} \geqslant   \frac{I_2(R, T)}{I_1(R, T)}  + O\lb T^{\beta-1 + (\log 4) B + o(1)}\rb.
\end{align}

Next, we will give a lower bound for $I_2(R, T)$.  We will use the fact that $\widehat \Phi(\xi)$ is positive, $r(n)$ is non-negative and $r(n)$ is completely multiplicative:
\begin{align*}
  I_2(R, T) &= \sum_{p^{\nu} \leqslant Y} \frac{\log p}{p^{\nu} } \text{Re}\lb \int_{-\infty}^{+\infty} \sum_{n, m \in \N} r(n)r(m) \lb\frac{n}{p^{\nu}m}\rb^{\i t}\Phi(\frac{t\log T}{T}) \d t \rb\\
  &\geqslant \sum_{p \leqslant X} \frac{\log p}{p } \int_{-\infty}^{+\infty} \sum_{\substack{k, m \in \N\\n = pk}} r(n)r(m) \lb\frac{n}{pm}\rb^{\i t}\Phi(\frac{t\log T}{T}) \d t\\
  &= \lb  \sum_{p \leqslant X} \frac{\log p}{p } r(p) \rb \int_{-\infty}^{+\infty} \sum_{\substack{k, m \in \N}} r(k)r(m) \lb\frac{k}{m}\rb^{\i t}\Phi(\frac{t\log T}{T}) \d t \,.
\end{align*}

Recall the definition of $I_1(R, T)$, then we obtain
\begin{align}\label{ratioI2I1}
    \frac{I_2(R, T)}{I_1(R, T)} \geqslant \sum_{p \leqslant X} \frac{\log p}{p } r(p). 
\end{align}

By the prime number theorem (for instance, see \cite{RS62}*{page 68}), we have
\begin{align}\label{ResultComp}
  \sum_{p \leqslant X}  \frac{\log p}{p} \lb 1 - \frac{p}{X}\rb = \log X -\gamma - \sum_{k = 2}^{\infty} \sum_p \frac{\log p}{p^k} - 1 + O\lb e^{-A\sqrt{\log X}}\rb.
\end{align}

According to \eqref{M2M1I2I1}, \eqref{ratioI2I1} and \eqref{ResultComp}, we obtain \begin{align}\label{boundM2M1}
 \frac{M_2(R, T)}{M_1(R, T)} \geqslant \log_2 T + \log_3 T + \log B  -\gamma - \sum_{k = 2}^{\infty} \sum_p \frac{\log p}{p^k} - 1 &+ O\lb e^{-A\sqrt{\log_2 T + \log_3 T + \log B}}\rb\\ \nonumber &+ O\lb T^{\beta-1 + (\log 4) B + o(1)}\rb.
\end{align}

Recall that  $B = \frac{1 - \beta}{\log 4}e^{-\frac{\epsilon}{2}}$. Let $C_1$ be a constant satisfying the following
\begin{align}\label{C1}
    C_1 = \log(1-\beta) - \log_2 4  -\gamma - \sum_{k = 2}^{\infty} \sum_p \frac{\log p}{p^k} - 1, 
\end{align}
then the claim in the theorem follows from \eqref{zeta: approx1}, \eqref{ratioM2M1} and \eqref{boundM2M1}.

\section{Proof of Theorem \ref{Zetasig=1: meas}}\label{Zetasig=1: measSection}
The proof is a slightly modification of the proof of Theorem \ref{Zetasig=1}.
Set $$B = \frac{1-\beta}{\log 4} \exp{\lb  -x +   T^{-A_{\beta}} +  f(T)\rb}, \quad \text{and} \quad f(T) = \exp{\lb -\lb\log_2 T\rb^{\frac{1}{3}} \rb}, $$
where $A_{\beta}$ is slightly smaller than the constant inside the big $O(\cdot)$ in \eqref{zeta: approx1}.

And define $J_x$, $\widetilde{J_x}$ as
\begin{align*}
    J_x &=\log_2 T + \log_3 T + C_1 -x + T^{-A_{\beta}} + \frac{1}{2}f(T),\\
    \widetilde{J_x} &=\log_2 T + \log_3 T + C_1 -x + T^{-A_{\beta}}.
\end{align*}

Note that for all sufficiently large $T$, we have $\beta - 1 + (\log 4) B < \beta - 1 + (1-\beta)\exp{(-\frac{1}{2}x)} < 0.$ Then by \eqref{boundM2M1}, for all sufficiently large $T$, we have
\begin{align}\label{M2M1Jx}
    \frac{M_2(R, T)}{M_1(R, T)} \geqslant J_x.
\end{align}

Next, define the three sets $V_x$, $W_x$ and $Z_x$  as follows
\begin{align*}
   V_x& = \{t \in [T^{\beta},  T]:\text{Re}\lb   \sum_{n\leqslant Y} \frac{\Lambda(n)}{n^{1 +\i t}}\rb \leqslant \widetilde{J_x}\},\\
W_x& = \{t \in [T^{\beta},  T]:\text{Re}\lb   \sum_{n\leqslant Y} \frac{\Lambda(n)}{n^{1 +\i t}}\rb > \widetilde{J_x}\}, \\
  Z_x& = \{t \in [T^{\beta},  T]: -\text{Re} \frac{\zeta^{\prime}}{\zeta}(1 + \i t) > \widetilde{J_x} - T^{-A_{\beta}}\}.
\end{align*}
 
By the definitions of $Z_x$ and $W_x$, $W_x$ is a subset of $Z_x$, so \,$\text{meas}(Z_x)\geqslant \text{meas}(W_x)$. And by the definitions of $V_x$ and $W_x$, we obtain
\begin{align*}
    M_2(R, T) = \int_{V_x} + \int_{W_x} \leqslant \widetilde{J_x}\cdot M_1(R, T) + \int_{W_x} \text{Re}\lb   \sum_{n\leqslant Y} \frac{\Lambda(n)}{n^{1 +\i t}}\rb |R(t)|^2 \Phi(\frac{t\log T}{T}) \d t.
\end{align*}

The above inequality together with \eqref{M2M1Jx} give the following 
\begin{align}\label{W_x: lower}
\frac{1}{2}  f(T)\cdot M_1(R, T) \leqslant \int_{W_x} \text{Re}\lb   \sum_{n\leqslant Y} \frac{\Lambda(n)}{n^{1 +\i t}}\rb |R(t)|^2 \Phi(\frac{t\log T}{T}) \d t .
\end{align}

By \cite{MV}*{Theorem 6.7}, the estimate $|\zeta^{\prime}(1+\i t)/\zeta(1 + \i t)| \leqslant A\log t$ holds for all $t\geqslant 10$. And recall that we have the estimate $|R(t)|^2 \leqslant T^{2 B + o(1)}$. Combining with \eqref{zeta: approx1}, we  obtain
\begin{align}\label{W_x: upper}
    \int_{W_x} \text{Re}\lb   \sum_{n\leqslant Y} \frac{\Lambda(n)}{n^{1 +\i t}}\rb |R(t)|^2 \Phi(\frac{t\log T}{T}) \d t \leqslant \lb A\log T\rb\, T^{2B + o(1)} \, \text{meas}(W_x).
\end{align}

Recall that $\beta - 1 + (\log 4) B <  0.$ So we have $M_1(R,T)\geqslant T^{1+B(2-\log 4) + o(1)}$, which follows from the facts that $I_1(R,T)\geqslant T^{1+B(2-\log 4) + o(1)}$ and $2M_1(R, T) = I_1(R, T) + O(T^{\beta + 2B +o(1)})$. Combining this lower bound with \eqref{W_x: lower} and \eqref{W_x: upper}, we deduce that
\begin{align*}
 \text{meas}(W_x)  &\geqslant \frac{\frac{1}{2}f(T) M_1(R, T)}{A\log T\cdot T^{2B + o(1)} } \\
 &\geqslant  T^{\lb1  + o\lb 1\rb\rb \lb1 - \lb \log 4\rb B \rb }.
\end{align*}

By the choice of $B$, we obtain
\[ \text{meas}(W_x)  \geqslant  T^{1-\lb1-\beta\rb e^{-x} + o\lb1\rb},\]
which gives the claim in the theorem since $\text{meas}(Z_x)\geqslant \text{meas}(W_x)$.

\section{Proof of Theorem  \ref{Lsig=1}}\label{ProofLsig=1}
Let $G_q$ be the group of Dirichlet characters $\chi$ (mod $q$). By a classical result due  to Gronwall and Titchmarsh (for instance, see \cite{MV}*{Theorem 11.3}), $\exists A > 0$, such that for all $\chi \in G_q$, the following region 
\[  \{s = \sigma + \i t: \sigma > 1 - \frac{A}{\log \lb q \lb|t| + 2\rb\rb}\} \]
contains no zero of $L(s, \chi)$, unless $\chi$ is a quadratic character, in this case $L(s, \chi)$ contains at most one zero (which is real) in  the above region.  If this exception occurs, we denote this particular character by $\chi_e$.

Define $G_q^{*} = G_q \setminus \{\chi_0, \chi_e\}$. Let $\chi \in G_q^{*}$. Take $c =\frac{1}{\log Y}$, $T = q^2$ and  $Y = \exp{\lb\lb \log qT\rb^2\rb}$. By Perron’s formula, we have
\begin{align*}
    \sum_{n\leqslant Y} \frac{\Lambda(n)\chi(n)}{n} =  \frac{1}{2\pi \i} \int_{c - \i T}^{c + \i T} 
-\frac{L^{\prime}( 1 +s, \chi)}{L(1 + s, \chi)} \frac{Y^s}{s} \d s  + O \lb \frac{(\log Y)^2}{T} + \frac{\log Y}{Y}\rb .
\end{align*}

 Similarly as  the proof of Theorem \ref{Zetasig=1}, we shift the line of integral to the line $\text{Re(s)} = -\delta = - \frac{A}{\log qT}$, to pick out the residue $-L^{\prime}(1, \chi)/L(1, \chi)$. The three  integrals $\int_{ c - \i  T}^{-\delta - \i T}$ , $\int_{ -\delta - \i  T}^{-\delta + \i T}$, and $\int_{ -\delta + \i  T}^{c+ \i T}$  in the contour can be bounded by using upper bounds for $L^{\prime}/L$ in the zero-free region (for instance, see \cite{MV}*{Theorem 11.14}). Proceeding similarly, we arrive at
\begin{align}\label{aprox: L1}
   - \frac{L^{\prime}}{L}(1, \chi) =  \sum_{n\leqslant Y} \frac{\Lambda(n)\chi(n)}{n} + O\lb q ^{-A} \rb\,, \quad \forall \chi \in G_q^{*}.
\end{align}

Next,  let $X = B\log q \log_2 q$, with $B = \frac{1 }{\log 4}e^{-\frac{\epsilon}{2}}$,  so that $-1 + (\log 4) B < 0$. And define $r(n)$ as in  section \ref{proofZetasig=1}. Define $R(\chi)$ as
\begin{align*}
    R(\chi) = \sum_{n\in \N} r(n) \chi(n), 
\end{align*}
then similar to $|R(t)|^2 \leqslant T^{2 B + o(1)}$ in section \ref{proofZetasig=1}, we have
\begin{align}\label{Rchi:Bd}
    \left| R(\chi) \right|^2 \leqslant q^{2B + o(1)}, \quad \forall \chi \in G_q.
\end{align}
And using a similar bound as \eqref{UpperBoundPartialSum}, we have
\begin{align}\label{prodRchi:Bd}
  \text{Re}\lb  \sum_{n\leqslant Y} \frac{\Lambda(n)\chi(n)}{n} \rb \left| R(\chi) \right|^2 \leqslant q^{2B + o(1)}, \quad \forall \chi \in G_q.
\end{align}

Next, define $ S_2(R, q)$, $ S_1(R, q)$, $ S_2^{*}(R, q)$  and $ S_1^{*}(R, q)$ as follows
\begin{align*}
 S_2(R, q) & = \sum_{\chi \in G_q }  \text{Re}\lb  \sum_{n\leqslant Y} \frac{\Lambda(n)\chi(n)}{n} \rb |R(\chi)|^2,\\
  S_1(R, q) & = \sum_{\chi \in G_q}  |R(\chi)|^2,\\
S_2^{*}(R, q) & = \sum_{\chi \in G_q^{*} }  \text{Re}\lb  \sum_{n\leqslant Y} \frac{\Lambda(n)\chi(n)}{n}\rb |R(\chi)|^2,\\
  S_1^{*}(R, q) & = \sum_{\chi \in G_q^{*}}  |R(\chi)|^2,
\end{align*}
then
\begin{align}\label{maxRe: S2S1Star}
  \max_{\chi \in G_q^{*}}  \text{Re} \sum_{n\leqslant Y} \frac{\Lambda(n)\chi(n)}{n} \geqslant \frac{S_2^{*}(R, q)}{S_1^{*}(R, q)}.
\end{align}

By \eqref{Rchi:Bd} and \eqref{prodRchi:Bd}, we find that
\begin{align*}
    S_2(R, q) &=  S_2^{*}(R, q) +  O\lb q^{2 B +  o(1)}\rb,\\
    S_1(R, q) &=  S_1^{*}(R, q) +  O\lb q^{2 B +  o(1)}\rb.
\end{align*}

By the orthogonality of  Dirichlet characters and the non-negativity of $r(n)$, we have
\begin{align*}
    S_1(R, q) &= \phi(q) \sum_{\substack{n, k\in \N\\n  \equiv k\ (\textrm{mod}\ q)\\ \gcd(k, q) =1 }}r(n)r(k)
    \\&\geqslant (q-1)\sum_{n \in \N} r^2(n)
    \\&\geqslant q^{1 + B(2 - \log 4) + o(1) },
\end{align*}
again  the last step follows from partial summation and the prime number theorem, which is similar to $I_1(R, T) \geqslant T^{1+B(2-\log 4) + o(1)}$ in Section \ref{proofZetasig=1}.

By the above estimates and the trivial fact that $S_1^{*}(R, q) \leqslant  S_1(R, q)$,  we obtain
\begin{align}\label{S2S1starToS2S1}
  \frac{S_2^{*}(R, q)}{S_1^{*}(R, q)} \geqslant  \frac{S_2(R, q)}{S_1(R, q)} + O\lb q^{-1+\lb\log 4\rb B +o(1)} \rb.
\end{align}

To establish a lower bound for $S_2(R, q)$, we once again make use of the orthogonality of Dirichlet characters and the non-negativity of $r(n)$:
\begin{align*}
  S_2(R, q) &=      \sum_{p^{\nu}\leqslant Y} \frac{\log p}{p^{\nu}} \text{Re} \lb \sum_{n, m \in \N}r(n)r(m) \sum_{\chi \in G_q }\chi(p^{\nu}n) \overline{\chi(m)}\rb\\
  &\geqslant \sum_{p \leqslant Y} \frac{\log p}{p} \phi(q) \sum_{\substack{n, m \in \N\\pn  \equiv m\ (\textrm{mod}\ q)\\ \gcd(m, q) =1 }}r(n)r(m) \\ &\geqslant \lb\sum_{p \leqslant X} \frac{\log p}{p} r(p) \rb \phi(q) \sum_{\substack{n, k\in \N\\pn  \equiv pk\ (\textrm{mod}\ q)\\ \gcd(pk, q) =1 }}r(n)r(k). 
\end{align*}

Since $q$ is prime, we have\footnote{The  assumption that $q$ is prime, is crucial. Without this assumption, the two sums  may not be equal. In \cite{AMMP}, there is no assumption that $q$  should be prime. This is not correct since there are counterexamples. See \textbf{Remark 3} in \cite{DY23} for the counterexamples.}
\[ \phi(q) \sum_{\substack{n, k\in \N\\pn  \equiv pk\ (\textrm{mod}\ q)\\ \gcd(pk, q) =1 }}r(n)r(k) = \phi(q) \sum_{\substack{n, k\in \N\\n  \equiv k\ (\textrm{mod}\ q)\\ \gcd(k, q) =1 }}r(n)r(k) = S_1(R, q).\]

The above identity and inequality lead to
\begin{align}\label{S2S1:ratio}
    \frac{S_2(R, q)}{S_1(R, q)} \geqslant \sum_{p \leqslant X} \frac{\log p}{p} r(p).
\end{align}

Recall that $X = B\log q \log_2 q$. By \eqref{ResultComp} and \eqref{S2S1starToS2S1}, we obtain
\begin{align}\label{LowerBound: S2S1Star}
  \frac{S_2^{*}(R, q)}{S_1^{*}(R, q)} \geqslant  &\log_2 q + \log_3 q + \log B - \gamma - \sum_{k = 2}^{\infty} \sum_p \frac{\log p}{p^k} - 1 \\\nonumber &+ O\lb e^{-A\sqrt{\log_2 q + \log_3 q + \log B}}\rb  + O\lb q^{-1+\lb\log 4\rb B +o(1)} \rb.
\end{align}

And recall that $B = \frac{1 }{\log 4}e^{-\frac{\epsilon}{2}}$. Let $C_2$ be a constant satisfying the following
\begin{align}\label{C2}
    C_2 =  - \log_2 4  -\gamma - \sum_{k = 2}^{\infty} \sum_p \frac{\log p}{p^k} - 1, 
\end{align}
then the claim in the theorem follows from \eqref{aprox: L1},  \eqref{maxRe: S2S1Star}, and \eqref{LowerBound: S2S1Star}.
\section{Proof of Theorem  \ref{Lsig=1: number}}
Again, the proof is a slightly modification of the proof of Theorem \ref{Lsig=1} and is similar as the proof of Theorem \ref{Zetasig=1: meas}.
Set $$B = \frac{1}{\log 4} \exp{\lb  -x + q^{-A} +  f(q)\rb}, \quad \text{and} \quad f(q) = \exp{\lb -\lb\log_2 q\rb^{\frac{1}{3}} \rb}, $$
where $A$ is slightly smaller than the constant inside the big $O(\cdot)$ in \eqref{aprox: L1}.

And define $J_x$, $\widetilde{J_x}$ as
\begin{align*}
    J_x &=\log_2 q + \log_3 q + C_2 -x +  q^{-A} + \frac{1}{2}f(q),\\
    \widetilde{J_x} &=\log_2 q + \log_3 q + C_2 -x +  q^{-A}.
\end{align*}

For all sufficiently large $q$, we have $ - 1 + (\log 4) B <  - 1 + \exp{(-\frac{1}{2}x)} < 0.$ Then by \eqref{LowerBound: S2S1Star}, for all sufficiently large $q$, we have
\begin{align}\label{S2S1Jx}
     \frac{S_2^{*}(R, q)}{S_1^{*}(R, q)} \geqslant J_x.
\end{align}

Next, define the three sets $V_x$, $W_x$ and $Z_x$  as follows
\begin{align*}
   V_x& = \{\chi \in G^{*}_q: \text{Re} \lb  \sum_{n\leqslant Y} \frac{\Lambda(n)\chi(n)}{n} \rb \leqslant \widetilde{J_x}\},\\
W_x& = \{\chi \in G^{*}_q: \text{Re}\lb  \sum_{n\leqslant Y} \frac{\Lambda(n)\chi(n)}{n} \rb > \widetilde{J_x}\}, \\
  Z_x& = \{\chi \in G^{*}_q: -\text{Re} \frac{L^{\prime}}{L}(1, \chi) > \widetilde{J_x} - q^{-A}\}.
\end{align*}
 
By the definitions of $Z_x$ and $W_x$, $W_x$ is a subset of $Z_x$, so \,$\#(Z_x)\geqslant \#(W_x)$. Moreover, by the definitions of $V_x$ and $W_x$, we obtain
\begin{align*}
    S_2^{*}(R, q) = \sum_{\chi \in V_x} + \sum_{\chi \in W_x} \leqslant \widetilde{J_x}\cdot  S_1^{*}(R, q) + \sum_{\chi \in W_x} \text{Re} \lb  \sum_{n\leqslant Y} \frac{\Lambda(n)\chi(n)}{n} \rb |R(\chi)|^2 .
\end{align*}

The above inequality together with \eqref{S2S1Jx} give the following 
\begin{align}\label{W_x: lower:q}
\frac{1}{2}  f(q)\cdot S_1^{*}(R, q) \leqslant  \sum_{\chi \in W_x} \text{Re} \lb  \sum_{n\leqslant Y} \frac{\Lambda(n)\chi(n)}{n} \rb |R(\chi)|^2.
\end{align}

By \cite{MV}*{Theorem 11.14}, the estimate $\left|L^{\prime}(1, \chi)/L(1, \chi)\right| \leqslant A\log q,$ holds for all $\chi \in G_q^{*}$. Recall the estimate $|R(\chi)|^2 \leqslant q^{2 B + o(1)}$. Combining with \eqref{aprox: L1}, we  have
\begin{align}\label{W_x: upper:q}
    \sum_{\chi \in W_x} \text{Re} \lb  \sum_{n\leqslant Y} \frac{\Lambda(n)\chi(n)}{n} \rb |R(\chi)|^2\leqslant \lb A\log q\rb\, q^{2B + o(1)} \cdot \#(W_x).
\end{align}

Since  $ - 1 + (\log 4) B <  0$,  we have $S_1^{*}(R, q) \geqslant q^{1+B(2-\log 4) + o(1)}$, which follows from the facts that $S_1(R, q) \geqslant q^{1 + B(2 - \log 4) + o(1) }$ and $S_1(R, q) =  S_1^{*}(R, q) +  O\lb q^{2 B +  o(1)}\rb$ in Section \ref{ProofLsig=1}. Combining this lower bound with \eqref{W_x: lower:q} and \eqref{W_x: upper:q}, we get
\begin{align*}
 \#(W_x)  &\geqslant \frac{\frac{1}{2}f(q) S_1^{*}(R, q)}{\lb A\log q \rb \cdot q^{2B + o(1)} } \\
 &\geqslant  q^{\lb1  + o\lb 1\rb\rb \lb1 - \lb \log 4\rb B \rb }.
\end{align*}

By the choice of $B$, we obtain
\[ \#(W_x)  \geqslant  q^{1- e^{-x} + o\lb1\rb},\]
which gives the claim in the theorem since $\#(Z_x)\geqslant \#(W_x)$.

\section{Proof of Theorem   \ref{Zeta: strip}}\label{section: strip: zeta}
We first establish the following lemma, which could be useful to approximate $\zeta^{\prime}(s)/\zeta(s)$ when combing with zero-density result of $\zeta(s)$. This lemma is similar to Lemma 8.2 in \cite{GS} and  Lemma 2.1 in \cite{GSZeta}, which are concerning approximations for $\log \zeta(s)$ and $\log L(s, \chi)$. 
\begin{lemma}\label{Zeta:approx:sig}
    Let $\y \geqslant 3$, and $t\geqslant \y + 3$. Let $\frac{1}{2} \leqslant \sigma_0 < 1$ and suppose that the rectangle $\{z: \sigma_0 <  \emph{Re}(z) \leqslant 1, |\emph{Im}(z) - t| \leqslant \y +2\}$ is free of zeros of $\zeta(z)$. Then for any $\sigma \in  (\sigma_0,\, 3]$ and  $ \xi \in [t-\y, t+\y ]$, we have 
    
    \begin{align*}
    \left|\frac{\zeta^{\prime}}{\zeta}(\sigma + \i \xi) \right|\ll \frac{\log t}{\sigma - \sigma_0}.
\end{align*} 

    Further, for $\sigma \in (\sigma_0,\, 1]$ and $\sigma_1 \in (\sigma_0,\, \sigma)$, we have
      \begin{align*}
   - \frac{\zeta^{\prime}}{\zeta}(\sigma + \i t)  = \sum_{n \leqslant \y}
\frac{\Lambda(n)}{n^{\sigma+\i t} } + O\lb    \frac{\log t}{\sigma_1 - \sigma_0} \y^{\sigma_1 -\sigma}  \log \frac{\y}{\sigma-\sigma_1} \rb.
\end{align*}
    \end{lemma}

\begin{proof}
  By  \cite{koukou}*{Lemma 8.2}, we have 
\begin{align*}
    \left|\frac{\zeta^{\prime}}{\zeta}(\sigma + \i \xi) \right|\leqslant \sum_{\substack{|\gamma - \xi|\leqslant 1,\\\zeta(\alpha+\i \gamma) = 0, \\ 0 < \alpha <1}} \frac{1}{\sigma - \alpha} + O\lb \log \xi\rb \ll \frac{\log t}{\sigma - \sigma_0},
\end{align*}
  where the last step follows from the inequality that $\alpha \leqslant \sigma_0$ by our assumption on the rectangle  and the inequality that $\#\{\alpha+\i \gamma: |\gamma - \xi|\leqslant 1,\, \zeta(\alpha+\i \gamma) = 0, \, 0 < \alpha <1\} \ll \log \xi \ll \log t$.

 Set $c = -\sigma + 1 + \frac{1}{\log \y}$. Again, by  Perron's formula, we find that
 \begin{align*}
      \sum_{n \leqslant \y}
\frac{\Lambda(n)}{n^{\sigma+\i t} } =    \frac{1}{2\pi \i} \int_{c - \i \y}^{c + \i \y} 
-\frac{\zeta^{\prime}( \sigma + \i t +s)}{\zeta( \sigma + \i t +s)} \frac{\y^s}{s} \d s  + O \lb \y^{-\sigma}\lb\log \y\rb^2\rb .
 \end{align*}

Note that if $\sigma_1 - \sigma \leqslant \text{Re}(s) \leqslant c$ and $|\text{Im}(s)|\leqslant \y$, then $\sigma_1 \leqslant \text{Re}(\sigma+\i t + s) \leqslant 1 + \frac{1}{\log \y}$ and $t -\y \leqslant \text{Im}(\sigma+\i t + s) \leqslant t+\y$.  By our assumption on the rectangle, we can move the line of integration from the line Re($s$) = $c$ to the line Re($s$) = $\sigma_1 - \sigma$, to pick out the residue:
 \begin{align}\label{ThreeInt: residue}
     \frac{1}{2\pi \i} \int_{c - \i \y}^{c + \i \y} 
-\frac{\zeta^{\prime}( \sigma + \i t +s)}{\zeta( \sigma + \i t +s)} \frac{\y^s}{s} \d s = -\frac{\zeta^{\prime}( \sigma + \i t )}{\zeta( \sigma + \i t )} +  \frac{1}{2\pi \i} \lb \int_{c-\i \y}^{\sigma_1 - \sigma  - \i \y} +  \int_{\sigma_1 - \sigma  - \i \y}^{\sigma_1 - \sigma  + \i \y}  + \int_{\sigma_1 - \sigma  + \i \y}^{c + \i \y} \rb.
 \end{align}
 
According to the first claim of the lemma, the first and third integrals on   the right side of \eqref{ThreeInt: residue} are bounded by
 \[\ll \int_{\sigma_1 -\sigma}^{1-\sigma + \frac{1}{\log \y}} \frac{\log t}{\sigma + x -\sigma_0} \frac{\y^x}{\y}\d x \ll \frac{\log t}{(\sigma_1 -\sigma_0)\log \y} \y^{-\sigma},\]
 and the second integral on the right side of \eqref{ThreeInt: residue} is bounded by
 \[\ll \int_{-\y}^{\y}  \frac{\log t}{\sigma_1 - \sigma_0} \frac{\y^{\sigma_1- \sigma}}{\sqrt{(\sigma_1 -\sigma)^2 + \xi^2}}\d \xi \ll \frac{\log t}{\sigma_1 - \sigma_0} \y^{\sigma_1 -\sigma}  \log \frac{\y}{\sigma-\sigma_1}. \]
 
 Now the second claim follows as well.
\end{proof}

Let $N(\sigma, T)$ denote the number of zeros of $\zeta(s)$ inside the rectangle $\{s: \text{Re}(s)\geqslant \sigma, 0< \text{Im}(s) \leqslant T\}$. For $\frac{1}{2} \leqslant \sigma_0 \leqslant 1$, $T \geqslant 2$, we have $N(\sigma_0, T)\ll T^{\frac{3(1-\sigma_0)}{2-\sigma_0}} \log^5 T$ by a classical result of Ingham \cite{Ingham}. Let $\epsilon \in (0, \sigma - \frac{1}{2}) $ be fixed. In Lemma \ref{Zeta:approx:sig}, we set $\sigma_0 = \sigma - \epsilon$, $\sigma_1 = \sigma_0 + \frac{1}{\log \y}$, and $\y = (\log T)^{\frac{20}{\epsilon}}$. Combining the zero-density result of Ingham, we  obtain
\begin{align}\label{zeta:equ:without:BT}
   - \frac{\zeta^{\prime}}{\zeta}(\sigma + \i t)  = \sum_{n \leqslant \y}
\frac{\Lambda(n)}{n^{\sigma + \i t} } + O\lb (\log T)^{-18} \rb, \quad \forall t \in [T^{\beta}, T]\setminus \B(T),
\end{align}
where $\B(T)$ is a ``bad''  set with measure satisfying
\begin{align}\label{meas: B(T)}
    \text{meas}\lb \B(T)\rb  \ll T^{\frac{3(1-\sigma + \epsilon)}{2-\sigma+\epsilon}}(\log T)^{5+\frac{20}{\epsilon}}.\end{align}

In the next, we take $X = B\log T \log _2 T$ and let $\Phi(t) = e^{ -\frac{t^2}{2} }$. Define $M_2(R, T)$, $M_1(R, T)$, $I_2(R, T)$ and $I_1(R, T)$ as follows ($R(t)$ to be defined later)
\begin{align*}
    M_2(R, T) &=\int_{T^\beta}^{T} \text{Re}\lb   \sum_{n\leqslant Y} \frac{\Lambda(n)}{n^{\sigma +\i t}} \rb |R(t)|^2 \Phi(\frac{t\log T}{T}) \d t,\\
    M_1(R, T) & = \int_{T^\beta}^{T}  |R(t)|^2 \Phi(\frac{t\log T}{T}) \d t,\\
     I_2(R, T) & =\int_{-\infty}^{+\infty} \text{Re}\lb   \sum_{n\leqslant Y} \frac{\Lambda(n)}{n^{\sigma +\i t}} \rb |R(t)|^2 \Phi(\frac{t\log T}{T}) \d t.\\
     I_1(R, T) & = \int_{-\infty}^{+\infty}  |R(t)|^2 \Phi(\frac{t\log T}{T}) \d t.
  \end{align*}
  then \begin{align}\label{ratioM2M1:sig}
      \max_{T^{\beta} \leqslant t \leqslant  T}\text{Re}\lb   \sum_{n\leqslant Y} \frac{\Lambda(n)}{n^{\sigma +\i t}} \rb \geqslant \frac{M_2(R, T)}{M_1(R, T)}.
  \end{align}

As in \cites{XY}, let  $r(n)$ be a completely multiplicative function and define its value at primes by 
 $$
 r(p)  = \begin{cases}
	1 - (\frac{p}{X})^{\sigma}, & \text{if\, $p \leqslant X$,}\\
            0, & \text{if\, $p > X$.}
		 \end{cases}
$$

Define $R(t) =    \sum_{n \in \N} r(n) n^{-\i t}.$ By the prime number theorem, we obtain  that $|R(t)|^2 \leqslant T^{2\sigma B + o(1)}$.
Instead of \eqref{UpperBoundPartialSum}, we would use  the following crude estimate 
\begin{align}\label{crude}
    \sum_{n\leqslant Y} \frac{\Lambda(n)}{n^{\sigma}}  \ll Y \ll T^{o(1)}. \end{align}
Then proceeding similarly as in Section \ref{proofZetasig=1},  we find that $2M_2(R, T) = I_2(R, T) + O(T^{\beta + 2 \sigma B +o(1)})$ and $2M_1(R, T) = I_1(R, T) + O(T^{\beta + 2\sigma B +o(1)})$.

By partial summation and the prime number theorem, we have (see \cite{DThesis}*{page 79}) \begin{align*}
    I_1(R, T) & = \int_{-\infty}^{+\infty} \sum_{\substack{k, m \in \N}} r(k)r(m) \lb\frac{k}{m}\rb^{\i t}\Phi(\frac{t\log T}{T}) \d t \\
    &\geqslant \sqrt{2\pi}\frac{T}{\log T}\sum_{n\in N} r^2(n)\\ &\geqslant T^{1+B\sigma\lb 1- c\lb\sigma\rb\rb + o(1)},
\end{align*}
 where
\begin{align*}
    c\lb\sigma\rb:\, = \int_0^1 \frac{t^\sigma}{2-t^\sigma} \d t.
\end{align*}

In order to make  error terms small, we require
\begin{align}\label{condition:B1}
    \beta + 2 \sigma B  < 1 +B\sigma\lb 1- c\lb\sigma\rb\rb.
\end{align}

 By the above estimate for $I_1(R, T)$ and the condition on $B$, we obtain
\begin{align}\label{M2M1I2I1:sig}
    \frac{M_2(R, T)}{M_1(R, T)} \geqslant   \frac{I_2(R, T)}{I_1(R, T)}  + O\lb T^{\beta + 2\sigma B - 1 - B\sigma\lb 1- c\lb\sigma\rb\rb + o(1)}\rb.
\end{align}

Similar to  \eqref{ratioI2I1}, we have
\begin{align*}
    \frac{I_2(R, T)}{I_1(R, T)} \geqslant \sum_{p \leqslant X}  \frac{\log p}{p ^{\sigma}} r(p). 
\end{align*}

When $\sigma \in [\frac{1}{2}, 1)$, by the prime number theorem, we have
\begin{align*}
  \sum_{p \leqslant X}  \frac{\log p}{p^{\sigma}} \lb 1 - \frac{p^{\sigma}}{X^{\sigma}}\rb = \lb \frac{\sigma}{1-\sigma}  + o\lb 1\rb \rb X^{1-\sigma}\,,\quad as\quad T \to \infty.
\end{align*}

Recall that $X = B \log T \log_2  T$, then we obtain \begin{align}\label{I2I1:sig}
 \frac{I_2(R, T)}{I_1(R, T)}\geqslant  \lb \frac{\sigma}{1-\sigma}  + o\lb 1\rb \rb B^{1-\sigma} \lb \log T \rb^{1-\sigma} \lb \log_2 T \rb^{1-\sigma}.
\end{align}

By \eqref{meas: B(T)}, for $M_2(R, T)$, the integration over the set $\B(T)$ is at most $\ll T^{2\sigma B +\frac{3(1-\sigma + \epsilon)}{2-\sigma+\epsilon} +o(1)} $. This contribution is negligible, if we require the following 
\begin{align}\label{condition:B2}
  2\sigma B +\frac{3(1-\sigma + \epsilon)}{2-\sigma+\epsilon}  < 1 +B\sigma\lb 1- c\lb\sigma\rb\rb.
\end{align}

Let $B$ be a positive constant satisfying \eqref{condition:B1} and \eqref{condition:B2}, then
the claim in the theorem follows\footnote{When $\sigma $ is away from $1/2$, for example, say $\sigma \geqslant 3/4 $ or $\sigma \geqslant 4/5$, then we can use other types of zero-density results to improve the constants. To keep the paper concise, we do not distinguish  these cases.} from  \eqref{zeta:equ:without:BT},    \eqref{ratioM2M1:sig}, \eqref{M2M1I2I1:sig} and \eqref{I2I1:sig}.

 \section{Proof of Theorem   \ref{L: strip}}\label{SectionProofLSrtip}
 Similar to Lemma \ref{Zeta:approx:sig}, we have the following Lemma.
\begin{lemma}\label{DL: appr: sig}
    Let $q$ be a prime larger than $3$. Let $\y\geqslant 3$, $-3q \leqslant t\leqslant 3q$,   and $\frac{1}{2} \leqslant \sigma_0 < 1$. Suppose that the rectangle $\{z: \sigma_0 <  \emph{Re}(z) \leqslant 1, |\emph{Im}(z) - t| \leqslant \y +2\}$ is free of zeros of $L(z, \chi)$, where $\chi$ is a non-principal character
 (mod $q$). Then for any $\sigma \in  (\sigma_0,\, 3]$ and  $ \xi \in [t-\y, t+\y ]$, we have 
    
    \begin{align*}
    \left|\frac{L^{\prime}}{L}(\sigma + \i \xi, \chi) \right|\ll \frac{\log q}{\sigma - \sigma_0}.
\end{align*} 

    Further, for $\sigma \in (\sigma_0,\, 1]$ and $\sigma_1 \in (\sigma_0,\, \sigma)$, we have
      \begin{align*}
   - \frac{L^{\prime}}{L}(\sigma + \i t, \chi)  = \sum_{n \leqslant \y}
\frac{\Lambda(n)}{n^{\sigma+\i t}} \chi(n)  + O\lb    \frac{\log q}{\sigma_1 - \sigma_0} \y^{\sigma_1 -\sigma}  \log \frac{\y}{\sigma-\sigma_1} \rb.
\end{align*}
    \end{lemma}
\begin{proof}
By  \cite{koukou}*{Lemma 11.4}, we have 
\begin{align*}
    \left|\frac{L^{\prime}}{L}(\sigma + \i \xi, \chi) \right|\leqslant \sum_{\substack{|\gamma - \xi|\leqslant 1,\\L(\beta+\i \gamma, \chi) = 0, \\ 0 < \beta <1}} \frac{1}{\sigma - \beta} + O\lb \log q\rb \ll \frac{\log q}{\sigma - \sigma_0}.
\end{align*}    

The following steps are the same as in the proof of Lemma \ref{Zeta:approx:sig}.
\end{proof}

Let $N(\sigma, T, \chi)$ denote the number of zeros of $L(s, \chi)$ inside the rectangle $\{s: \text{Re}(s)\geqslant \sigma,  |\text{Im}(s)| \leqslant T\}$. For $\frac{1}{2} \leqslant \sigma_0 \leqslant 1$, $T \geqslant 2$, we have $\sum_{\chi \in G_q} N(\sigma_0, T, \chi)\ll \lb qT\rb^{\frac{3(1-\sigma_0)}{2-\sigma_0}} \log^{14} \lb qT\rb$, by a zero-density result of Montgomery \cite{Mon71}*{Theorem 12.1}. Let $\epsilon \in (0, \sigma - \frac{1}{2}) $ be fixed.  Set $t = 0$, $\sigma_0 = \sigma - \frac{1}{2}\epsilon$, $\sigma_1 = \sigma_0 + \frac{1}{\log \y}$ and $\y = \lb \log q \rb ^\frac{20}{\epsilon}$ in Lemma \ref{DL: appr: sig}.  Coming the zero-density estimate of Montgomery (by taking  $T = Y + 2$),  we  obtain
\begin{align}\label{DL:approx:strip}
     - \frac{L^{\prime}}{L}(\sigma, \chi)  = \sum_{n \leqslant \y}
\frac{\Lambda(n)}{n^{\sigma}} \chi(n)  + O\lb    \lb\log q \rb^{-8} \rb,  \quad\forall \chi \in G_q^{*}\setminus \text{Bad}_{\sigma}(q),
\end{align}
where $\text{Bad}_{\sigma}(q)$ is a set of of ``bad" characters with cardinality satisfying
\begin{align}\label{cardniality: bad}
    \# \text{Bad}_{\sigma}(q) \ll q^{\frac{3(1-\sigma + \frac{1}{2}\epsilon)}{2-\sigma+\frac{1}{2}\epsilon}}  \lb\log q\rb^{O\lb1\rb}.
\end{align}

Let $X = B\log q\log_2 q$ and let $r(n)$ be the same function in Section \ref{section: strip: zeta}. Define $R(\chi) = \sum_{n \in \N} r(n)\chi(n) $, then again by the prime number theorem, we have
\begin{align}\label{trivialUpper: R(chi): strip}
|R(\chi)|^2 \leqslant q^{2\sigma B + o(1)}.\end{align} Define four sums as follows
\begin{align*}
 S_2(R, q) & = \sum_{\chi \in G_q }  \lb \text{Re} \sum_{n \leqslant \y}
\frac{\Lambda(n)}{n^{\sigma}} \chi(n)  \rb |R(\chi)|^2,\\
  S_1(R, q) & = \sum_{\chi \in G_q}  |R(\chi)|^2,\\
S_2^{*}(R, q) & = \sum_{\chi \in G_q^{*}\setminus \text{Bad}_{\sigma}(q) }  \lb \text{Re} \sum_{n \leqslant \y}
\frac{\Lambda(n)}{n^{\sigma}} \chi(n)  \rb |R(\chi)|^2,\\
  S_1^{*}(R, q) & = \sum_{\chi \in G_q^{*}\setminus \text{Bad}_{\sigma}(q)}  |R(\chi)|^2,
\end{align*}
then
\begin{align}\label{maxRe: S2S1Star:strip}
  \max_{\chi \in G_q^{*}\setminus \text{Bad}_{\sigma}(q)}  \text{Re} \sum_{n \leqslant \y}
\frac{\Lambda(n)}{n^{\sigma}} \chi(n) \geqslant \frac{S_2^{*}(R, q)}{S_1^{*}(R, q)}.
\end{align}

By \eqref{crude},  \eqref{cardniality: bad} and  \eqref{trivialUpper: R(chi): strip}, we have\begin{align*}
    S_2(R, q) &=  S_2^{*}(R, q) +  O\lb q^{2\sigma B + \frac{3(1-\sigma + \frac{1}{2}\epsilon)}{2-\sigma+\frac{1}{2}\epsilon}  + o(1)}\rb.
\end{align*}

Again by partial summation and the prime number theorem, we have 
\begin{align*}
  S_1(R, q) = \phi(q) \sum_{\substack{n, k\in \N\\n  \equiv k\ (\textrm{mod}\ q)\\ \gcd(k, q) =1 }}r(n)r(k)  \geqslant  (q-1)\sum_{n\in N} r^2(n)\geqslant q^{1+B\sigma\lb 1- c\lb\sigma\rb\rb + o(1)}.
\end{align*}

By the above estimates and $S_1^{*}(R, q) \leqslant  S_1(R, q)$,  we obtain
\begin{align}\label{reduce:to:S2S1:strip}
    \frac{S_2^{*}(R, q)}{S_1^{*}(R, q)} \geqslant \frac{S_2(R, q)}{S_1(R, q)} + O\lb q^{2\sigma B + \frac{3(1-\sigma + \frac{1}{2}\epsilon)}{2-\sigma+\frac{1}{2}\epsilon} -1-B\sigma\lb 1- c\lb\sigma\rb\rb + o(1)}\rb.
\end{align}

We require 
\begin{align}\label{condition:BLStrip}
    2\sigma B + \frac{3(1-\sigma + \frac{1}{2}\epsilon)}{2-\sigma+\frac{1}{2}\epsilon} -1 - B\sigma\lb 1- c\lb\sigma\rb\rb < 0,
\end{align}
so the error term on the right side of \eqref{reduce:to:S2S1:strip} is negligible.
And similar to \eqref{S2S1:ratio}, we have
\begin{align}\label{S2S1:ratio:strip}
  \frac{S_2(R, q)}{S_1(R, q)}  \geqslant \sum_{p \leqslant X}  \frac{\log p}{p ^{\sigma}} r(p) =  \lb \frac{\sigma}{1-\sigma}  + o\lb 1\rb \rb B^{1-\sigma} \lb \log q \rb^{1-\sigma} \lb \log_2 q \rb^{1-\sigma}.
\end{align}
Let $B$ be a positive constant satisfying \eqref{condition:BLStrip}, then the claim in the theorem follows from \eqref{DL:approx:strip}, \eqref{maxRe: S2S1Star:strip}, \eqref{reduce:to:S2S1:strip},  and \eqref{S2S1:ratio:strip}.

 \section{Proof of Theorem   \ref{Zeta=1: theta}}\label{ThmzetaTheta}
As in Section \ref{section: strip: zeta}, we use Lemma \ref{Zeta:approx:sig} and the zero-density estimate of Ingham (by set $\sigma_0 = 1 - \epsilon$, $\sigma_1 = \sigma_0 + \frac{1}{\log \y}$, and $\y = \lb \log T \rb ^\frac{20}{\epsilon}$) to find that\begin{align}\label{Eqzeta:approx:theta}
      e^{-\i \theta} \frac{\zeta^{\prime}}{\zeta}(1 + \i t)  = - e^{-\i \theta} \sum_{p \leqslant \y}
\frac{\log p}{p^{1 + \i t}}   + O\lb  1 \rb,  \quad\quad \forall |t| \in [T^{\beta}, T]\setminus \B(T),
\end{align}
where $\B(T)$ is a ``bad''  set with measure satisfying
\begin{align}\label{ThetaMeasB(T)}
    \text{meas}\lb \B(T)\rb  \ll T^{\frac{3 \epsilon}{1 +\epsilon}}(\log T)^{5+\frac{20}{\epsilon}}.\end{align}
    
Since $Y$ is a power of $\log T$, we find that 
\begin{align}\label{ZetaThetaUpperBoundPartialSum}
\left| - e^{-\i \theta} \sum_{p \leqslant \y}
\frac{\log p}{p^{1 + \i t}}\right| \leqslant  \sum_{p \leqslant \y}
\frac{\log p}{p} = \log Y + O(1) \ll \log_2 T.
\end{align}

Set $X = \frac{\log N}{\lb \log_2 N \rb^5}$. Let  $r(n)$ be a completely multiplicative function and define its value at primes by 
 $$
 r(p)  = \begin{cases}
	- e^{-\i \theta}\lb 1 - \frac{1}{(\log_2 N)^2} \rb , & \text{if\, $p \leqslant X$,}\\
            0, & \text{if\, $p > X$.}
		 \end{cases}
$$

Then we have the following asymptotic formula
\begin{align}\label{Ratior}
  \sum_{\substack{n\leqslant \frac{N}{X}}}\left|r(n)\right|^2 = \lb 1 + O\lb \frac{1}{\log N}\rb \rb\sum_{n \leqslant N}
 \left|r(n)\right|^2, \quad \forall N \geqslant 100,
\end{align} 
which follows from applications of  Rankin's trick. For instance,   we could directly deduce \eqref{Ratior}  from the following result of Hough. See \cite{Hough}*{page 100, page 107}.

\begin{lemma}[Adapted
from  \cite{Hough}*{Lemma 5.3}]\label{Hough}

Let $f_k (n)$ be a sequence of non-negative completely multiplicative functions and let $N_k \to \infty$ be  
a growing sequence of parameters. Define 
\[\alpha_k = \frac{(\log_2 N_k)^2}{\log N_k}\,.\]
Assume that $f_k(p)p^{\alpha_k} < 1$ for all primes $p$ and all sufficiently large $k$. And suppose that for  sufficiently large $k$, we have
\[\sum_p \log p \frac{f_k(p)}{1 - f_k(p)}< \log N_k - \frac{\log N_k}{\log_2 N_k}\]
and \[ \sum_p\sum_{n > \frac{\log N_k}{ (\log p) \left(\log_2 N_k\right)^4  }} \frac{\left(f_k(p)p^{\alpha_k}\right)^n}{n} = O(1)\,. \]
Then for sufficiently large  $k$, we have
\[ \sum_{ n \leqslant N_k}f_k(n) = \left(1+O\left(\frac{1}{\log N_k}\right)\right) \sum_{ n = 1}^{\infty}f_k(n)\,.  \]

\end{lemma}

Let  $N = \lfloor T^{\kappa} \rfloor$ (here $\kappa$ is a fixed positive number smaller than $1$, to be chosen later) and define $R(t) = \sum_{n \leqslant N} r(n) n^{-\i t} $,  then by Cauchy's inequality,  we have the following trivial bound
\begin{align}\label{trivialUpper: R(t): theta}
|R(t)|^2 \leqslant N \sum_{n \leqslant N}|r(n)|^2 \leqslant N^2 \leqslant T^{2\kappa}.\end{align}

Define $M_2(R, T)$, $M_1(R, T)$, $I_2(R, T)$ and $I_1(R, T)$ as follows
\begin{align*}
    M_2(R, T) &=\int_{T^{\beta} \leqslant |t| \leqslant  T} \text{Re}\lb  - e^{-\i \theta} \sum_{p \leqslant \y}
\frac{\log p}{p^{1 + \i t}} \rb |R(t)|^2 \Phi(\frac{t\log T}{T}) \d t,\\
    M_1(R, T) & = \int_{T^{\beta} \leqslant |t| \leqslant  T}  |R(t)|^2 \Phi(\frac{t\log T}{T}) \d t,\\
     I_2(R, T) & =\int_{-\infty}^{+\infty} \text{Re}\lb  - e^{-\i \theta} \sum_{p \leqslant \y}
\frac{\log p}{p^{1 + \i t}} \rb |R(t)|^2 \Phi(\frac{t\log T}{T}) \d t,\\
     I_1(R, T) & = \int_{-\infty}^{+\infty}  |R(t)|^2 \Phi(\frac{t\log T}{T}) \d t,
  \end{align*}
  then \begin{align}\label{ratioM2M1Theta}
      \max_{T^{\beta} \leqslant |t| \leqslant  T}\text{Re}\lb   - e^{-\i \theta} \sum_{p \leqslant \y}
\frac{\log p}{p^{1 + \i t}}\rb \geqslant \frac{M_2(R, T)}{M_1(R, T)}.
  \end{align}

To compute $ I_2(R, T)$, we will split it into two sums. For the diagonal terms,  namely, $pn = m$, we make use of the facts that $r(n)$ is completely multiplicative and $ \arg (-\overline{r(p)}) = \theta $. Then we obtain
\begin{align*}
  I_2(R, T) &=  \text{Re}\lb - e^{-\i \theta} \sum_{p \leqslant Y} \frac{\log p}{p }\int_{-\infty}^{+\infty} \sum_{n,\, m \leqslant  N} r(n) \overline{r(m)}  \lb\frac{m}{pn}\rb^{\i t}\Phi\lb\frac{t\log T}{T}\rb \d t \rb\\
  &= \sqrt{2\pi} \frac{T}{\log T} \text{Re}\lb - e^{-\i \theta} \sum_{p \leqslant Y} \frac{\log p}{p } \sum_{n,\, m \leqslant N} r(n)\overline{r(m)}  \Phi\lb \frac{T}{\log  T} \log \frac{m}{pn} \rb  \rb\\
   &= \sqrt{2\pi} \frac{T}{\log T}\sum_{p \leqslant X} \frac{\log p}{p} \left|r(p)\right| \sum_{\substack{n\leqslant \frac{N}{p}}}\left|r(n)\right|^2  + \Er,
\end{align*}
where $\Er$ is defined by
\begin{align*}
    \Er:\, = \sqrt{2\pi} \frac{T}{\log T} \text{Re}\lb - e^{-\i \theta} \sum_{p \leqslant Y} \frac{\log p}{p } \sum_{ \substack{n,\, m \leqslant N\\ pn \neq m}} r(n)\overline{r(m)}  \Phi\lb \frac{T}{\log  T} \log \frac{m}{pn} \rb  \rb.
\end{align*}

Since $p \leqslant Y$, and $m, n \leqslant N$, if $pn \neq m$, then $\left| \log \frac{m}{pn} \right|\geqslant (YN)^{-1}$.   Since $Y$ is a power of $\log T$ and $N  \leqslant T^{\kappa} $, we find that
\[ \Phi\lb \frac{T}{\log  T} \log \frac{m}{pn} \rb \leqslant e^{- T^A}, \quad \forall n, m \leqslant N, p\leqslant Y, \,\text{and}\,\, pn \neq m. \]

Note that $|r(n)| \leqslant 1$. Then using the above estimate, we obtain
\[ |\Er| \ll e^{- T^A}.\]

So we obtain 
\begin{align*}
    I_2(R, T) = \sqrt{2\pi} \frac{T}{\log T}\sum_{p \leqslant X} \frac{\log p}{p} \left|r(p)\right| \sum_{\substack{n\leqslant \frac{N}{p}}}\left|r(n)\right|^2  +   O\lb e^{- T^A}\rb.
\end{align*}

And similarly, we have
\begin{align*}
    I_1(R, T) = \sqrt{2\pi} \frac{T}{\log T} \sum_{n \leqslant N}|r(n)|^2  + O\lb e^{- T^A}\rb.
\end{align*}

By the above two formulas and  \eqref{Ratior}, we deduce that
\begin{align}\label{I2I1:theta}
    \frac{I_2(R, T)}{I_1(R, T)} \geqslant  \lb 1 + O\lb \frac{1}{\log T}\rb \rb \sum_{p \leqslant X} \frac{\log p}{p} \left|r(p)\right| + O\lb 1\rb = \log_2 T + O\lb\log_3 T \rb. 
\end{align}

And since $r(1) = 1$, we have $I_1(R, T) \gg \frac{T}{\log T}$. By the estimates \eqref{ZetaThetaUpperBoundPartialSum} and \eqref{trivialUpper: R(t): theta}, we deduce that $M_2(R, T) = I_2(R, T) + O\lb T^{2\kappa + \beta + o(1)}\rb$ and $M_1(R, T) = I_1(R, T) + O\lb T^{2\kappa + \beta + o(1)}\rb$.
Combining with the trivial bound $M_1(R, T) \leqslant I_1(R, T)$,   we obtain
\begin{align}\label{M2M1I2I1:theta}
    \frac{M_2(R, T)}{M_1(R, T)} \geqslant   \frac{I_2(R, T)}{I_1(R, T)}  + O\lb T^{ 2\kappa + \beta  - 1 + o(1)}\rb.
\end{align}

Let $\kappa$   be a fixed positive number such that $2\kappa + \beta  - 1   < 0$. By \eqref{ThetaMeasB(T)}, for $M_2(R, T)$, the integration over the set $\B(T)$ is at most $\ll T^{2\kappa +\frac{3\epsilon}{1 + \epsilon} +o(1)} $. This contribution is negligible, if we let $\epsilon$ be a  fixed positive number satisfying the following 
\begin{align*}
  2\kappa +\frac{3\epsilon}{1 + \epsilon}  < 1.
\end{align*}
Then the claim in the theorem follows from \eqref{Eqzeta:approx:theta}, \eqref{ratioM2M1Theta}, \eqref{I2I1:theta}, and \eqref{M2M1I2I1:theta}.

 \section{Proof of Theorem   \ref{Lsig=1: theta}}\label{ThmLTheta}

As in Section \ref{SectionProofLSrtip}, we use Lemma \ref{DL: appr: sig} and the zero-density estimate of Montgomery (by set $\sigma_0 = 1 - \epsilon$, $\sigma_1 = \sigma_0 + \frac{1}{\log \y}$, and $\y = \lb \log q \rb ^\frac{20}{\epsilon}$) to find that\begin{align}\label{DL:approx:theta}
      e^{-\i \theta} \frac{L^{\prime}}{L}(1, \chi)  = - e^{-\i \theta} \sum_{p \leqslant \y}
\frac{\log p}{p} \chi(p)  + O\lb  1 \rb,  \quad\forall \chi \in G_q^{*}\setminus \text{Bad}(q),
\end{align}
where $\text{Bad}(q)$ is a set of of ``bad" characters with cardinality satisfying
\begin{align}\label{CardnialityOneBad}
    \# \text{Bad}(q) \ll q^{\frac{3\epsilon}{1+\epsilon}}  \lb\log q\rb^{O\lb1\rb}.
\end{align}

Note that 
\begin{align}\label{DLThetaUpperBoundPartialSum}
\left| - e^{-\i \theta} \sum_{p \leqslant \y}
\frac{\log p}{p} \chi(p)\right| \leqslant  \sum_{p \leqslant \y}
\frac{\log p}{p} \ll \log Y \ll \log_2 q.
\end{align}

Let $r(n)$ be the same function in Section \ref{ThmzetaTheta}, with $N = \lfloor q^{\frac{1}{24}} \rfloor$. Define $R(\chi) = \sum_{n \leqslant N} r(n)\chi(n) $. Then by Cauchy's inequality,  we have the trivial bound
\begin{align}\label{trivialUpper: R(chi): theta}
|R(\chi)|^2 \leqslant N \sum_{n \leqslant N}|r(n)|^2 \leqslant N^2 \leqslant q^{\frac{1}{12}}.\end{align} 

Define four sums as follows
\begin{align*}
 S_2(R, q) & = \sum_{\chi \in G_q }  \text{Re}\lb  - e^{-\i \theta} \sum_{p \leqslant \y}
\frac{\log p}{p} \chi(p)  \rb |R(\chi)|^2,\\
  S_1(R, q) & = \sum_{\chi \in G_q}  |R(\chi)|^2,\\
S_2^{*}(R, q) & = \sum_{\chi \in G_q^{*}\setminus \text{Bad}(q) } \text{Re}\lb   - e^{-\i \theta} \sum_{p \leqslant \y}
\frac{\log p}{p} \chi(p) \rb |R(\chi)|^2,\\
  S_1^{*}(R, q) & = \sum_{\chi \in G_q^{*}\setminus \text{Bad}(q)}  |R(\chi)|^2,
\end{align*}
then
\begin{align}\label{maxRe: S2S1Star:theta}
  \max_{\chi \in G_q^{*}\setminus \text{Bad}(q)}  \text{Re}\lb   - e^{-\i \theta} \sum_{p \leqslant \y}
\frac{\log p}{p} \chi(p) \rb  \geqslant \frac{S_2^{*}(R, q)}{S_1^{*}(R, q)}.
\end{align}

By \eqref{CardnialityOneBad}, \eqref{DLThetaUpperBoundPartialSum},   and  \eqref{trivialUpper: R(chi): theta}, we have\begin{align*}
    S_2(R, q) &=  S_2^{*}(R, q) +  O\lb q^{\frac{1}{12} + \frac{3\epsilon}{1+\epsilon} + o(1)}\rb.
\end{align*}

Since $N \leqslant  q^{\frac{1}{24}}$,  by  the orthogonality of Dirichlet
characters, we have
\begin{align*}
    S_1(R, q) = \sum_{n, m \leqslant N}  r( n) \overline{r(m)}\sum_{\chi \in G_q} \chi(n) \overline{\chi(m)} = \phi(q)\sum_{n \leqslant N}
 \left|r(n)\right|^2 \geqslant q - 1.
\end{align*}

By the above estimates and $S_1^{*}(R, q) \leqslant  S_1(R, q)$,  we obtain
\begin{align}\label{reduce:to:S2S1:theta}
    \frac{S_2^{*}(R, q)}{S_1^{*}(R, q)} \geqslant \frac{S_2(R, q)}{S_1(R, q)} + O\lb q^{\frac{1}{12} + \frac{3\epsilon}{1+\epsilon} - 1 +  o(1)}\rb.
\end{align}

By  the orthogonality of Dirichlet
characters, we obtain
\begin{align*}
    S_2(R, q) &= \text{Re}\lb - e^{-\i \theta} \sum_{p \leqslant \y}\frac{\log p}{p} \sum_{n, m \leqslant N}  r( n) \overline{r(m)}\sum_{\chi \in G_q} \chi(p n) \overline{\chi(m)}\rb\\
  &=  \text{Re}\lb - e^{-\i \theta} \sum_{p \leqslant Y} \frac{\log p}{p} \phi(q) \sum_{\substack{n, m \leqslant N\\pn  \equiv m\ (\textrm{mod}\ q)\\ \gcd(m,\, q) =1 }}r(n)\overline{r(m)} \rb\\
  &= \text{Re}\lb - e^{-\i \theta} \phi(q) \sum_{p \leqslant Y} \frac{\log p}{p} \overline{r(p)} \sum_{\substack{n\leqslant \frac{N}{p}}}\left|r(n)\right|^2 \rb,
\end{align*}
where the last step follows from the fact that if $p \leqslant Y$, $n, m\leqslant N$,   then $pn\equiv m\ (\textrm{mod}\ q)$ with $\gcd(m, q) =1$ is equivalent to $pn = m$. By the definition of $r(n)$ at primes, we have 
\begin{align*}
    S_2(R, q) =\phi(q) \sum_{p \leqslant X} \frac{\log p}{p} \left|r(p)\right| \sum_{\substack{n\leqslant \frac{N}{p}}}\left|r(n)\right|^2.
\end{align*}

Then by \eqref{Ratior},  we obtain 
\begin{align*}
    \frac{S_2(R, q)}{S_1(R, q)} \geqslant  \lb 1 + O\lb \frac{1}{\log q}\rb \rb \sum_{p \leqslant X} \frac{\log p}{p} \left|r(p)\right|  = \log_2 q  + O\lb\log_3 q \rb. 
\end{align*}

Let $\epsilon$ be a sufficiently small fixed positive number, then the claim in the theorem follows from \eqref{DL:approx:theta}, \eqref{maxRe: S2S1Star:theta}, and \eqref{reduce:to:S2S1:theta}.

\section{Further Discussions,  Conjectures and Problems}\label{sec:res}
First, we would like to discuss the relationship between the maximal values of the logarithmic derivatives of the Dirichlet  $L$-functions and the maximal values of the Dirichlet  $L$-functions and their derivatives.  In \cite{DY23}*{Theorem 2}, we established that\begin{align*}
   \max_{ \substack{  \chi \neq \chi_0 \\ \chi(\text{mod}\, q)}} \left|L^{\prime}(1, \chi) \right| \geqslant \left(e^{\gamma}+o\left(1\right)\right) \left(\log_2 q\right)^{2},\,\, \quad \text{as prime}\,\quad q \to \infty.
\end{align*} 

On the other hand, assuming the truth of a conjecture of Granville-Soundararajan (conjecture 1 in \cite{GS}) on character sums, one can (see \cite{DY23}*{Theorem 6}) deduce that  $$ \left|L(1, \chi) \right|\leqslant \left(e^{\gamma}+o\left(1\right)\right) \left(\log_2 q\right),\quad \text{as}\,\quad q \to \infty,$$   for all non-principal characters $\chi$ (mod $q$). By the conditional upper bound and the unconditional lower bound, one can immediately obtain
\begin{align*}
   \max_{ \substack{  \chi \neq \chi_0 \\ \chi(\text{mod}\, q)}} \left|\frac{L^{\prime}}{L}(1, \chi)\right| \geqslant \left(1 +o\left(1\right)\right) \left(\log_2 q\right),\,\, \quad \text{as prime}\,\quad q \to \infty,
\end{align*} 
which already reproduces  Lamzouri's result. Although this conditional result is weaker, it suggests a potential relationship between the maximum of the logarithmic derivative and the maximums  of the derivative and the original function. By looking at the Dirichlet polynomials approximations for $ L^{\prime}(1, \chi)/L(1, \chi)$,  $L(1, \chi)$ and
$L^{\prime} (1, \chi)$, we think it should be true that $|L(1, \chi)|$ is large if and only if $|L^{\prime}(1, \chi)|$  is large. It should also be true that if both $|L^{\prime} (1, \chi)|$  and $|L(1, \chi)|$ are large, then $|L^{\prime}(1, \chi)/L(1, \chi)|$  is large.   Moreover, asymptotically, when varying $\chi$,  $ L^{\prime}(1, \chi)/L(1, \chi)$ should be able to be large in any direction in the complex plane. But we do not think $L(1, \chi)$ or $L^{\prime} (1, \chi)$ could be large in any direction in the complex plane since in such cases products of primes are involved so one cannot make each summand point towards  any given direction without loss,  unlike the case of $ L^{\prime}(1, \chi)/L(1, \chi)$ which essentially only have isolated primes involved. This phenomenon should  hold for $\sigma \in [1/2, 1)$ as well.  And philosophically speaking, we believe the following principles should be true.
\begin{Principle}
   For all large prime $q$, and for any $\theta \in [0, 2\pi]$, there exists a  nonprincipal Dirichlet character $\chi$ (mod $q$) such that $\chi(p) \approx e^{\i\theta}, \quad \forall p \ll \log q \log_2 q.$
\end{Principle}

\begin{Principle}
   For all large $T$, and for any $\theta \in [0, 2\pi]$, there exists  $t\in [T, 2T]$ such that 
\begin{align*}
  p^{\i t} \approx e^{\i\theta}, \quad \forall p \ll \log T \log_2 T.  
\end{align*}
\end{Principle}

On the other hand, it is widely believed that $L(\frac{1}{2}, \chi)$ is not zero for all primitive characters $\chi$. The best result for the quadratic case is due to Soundararajan \cite{SoundNonVan}.  Combining the above discussion, we propose the following  conjecture.

\begin{conjecture}\label{CJDL1}
 As prime $q \to \infty$, we have \begin{align*}
\max_{ \substack{  \chi \neq \chi_0 \\ \chi(\emph{mod}\, q)}} \emph{Re} \lb e^{-\i \theta} \frac{L^{\prime}}{L}\lb\sigma, \chi\rb\rb
\sim \lb  \max_{ \substack{  \chi \neq \chi_0 \\ \chi(\emph{mod}\, q)}}  \left|L^{\prime}(\sigma, \chi) \right| \rb \Big/ \lb  \max_{ \substack{  \chi \neq \chi_0 \\ \chi(\emph{mod}\, q)}}  \left|L(\sigma, \chi) \right| \rb,
\end{align*}
uniformly for all $\theta \in [0, 2\pi]$ and all  $\sigma \in [\frac12, 1]$.
  \end{conjecture}

And we have similar results for  $ |\zeta^{\prime}(1 + \i t)|$ in \cite{DY 23}. For early research on large values of  $ |\zeta^{\prime}(\sigma + \i t)|$,  see \cites{Ka, DY21, logGCD, DW}. We propose the following similar conjecture for  the Riemann zeta function. \begin{conjecture}\label{CJDL1Z}
Let $\epsilon \in (0, \frac12)$ be fixed. Then as $T \to \infty$, we have \begin{align*}
\max_{T \leqslant t \leqslant  2T}\emph{Re} \lb e^{-\i \theta}  \frac{\zeta^{\prime}}{\zeta}\lb \sigma + \i t\rb \rb \sim \frac{ \max_{T \leqslant t \leqslant  2T} \left| \zeta^{\prime}(\sigma + \i t)\right| }{\max_{T \leqslant t \leqslant  2T}\left|\zeta(\sigma + \i t)\right|},
\end{align*}
uniformly for all $\theta \in [0, 2\pi]$ and all  $\sigma \in [\frac12 + \epsilon, 1]$. 
  \end{conjecture}

Compared to Littlewood's conditional upper bounds, we think our lower bounds are more close to the truth and we have the following four conjectures. 

\begin{conjecture}\label{CJDL1Z1}
 As  $T \to \infty$, we have \begin{align*}
\max_{T \leqslant t \leqslant  2T}\emph{Re} \lb e^{-\i \theta }\frac{\zeta^{\prime}}{\zeta}\lb1 + \i t\rb \rb = \log_2 T + \log_3 T + O(1),
\end{align*}
 uniformly for all $\theta \in [0, 2\pi]$.
\end{conjecture}

\begin{conjecture}\label{CJDL2}
As prime $ q \to \infty$, we have \begin{align*}
\max_{ \substack{  \chi \neq \chi_0 \\ \chi(\emph{mod}\, q)}} \emph{Re} \lb e^{-\i \theta} \frac{L^{\prime}}{L}\lb 1, \chi\rb\rb = \log_2 q + \log_3 q + O(1),
\end{align*}
 uniformly for all $\theta \in [0, 2\pi]$.
\end{conjecture}

\begin{conjecture}
Let $\sigma \in (\frac12, 1)$ be fixed. Then for all large $t$, we have \begin{align*}
\left|\frac{\zeta^{\prime}}{\zeta}\lb \sigma + \i t\rb \right|\ll_{\sigma} \lb\log t\rb^{1-\sigma}\lb\log_2 t\rb^{1-\sigma}.
\end{align*}
\end{conjecture}

\begin{conjecture}
Let $\sigma \in (\frac12, 1)$ be fixed. Then for all large integer $q$ and all primitive characters $\chi (\emph{mod}\, q)$, we have \begin{align*}
\left|\frac{L^{\prime}}{L}\lb\sigma, \chi\rb \right|\ll_{\sigma} \lb\log q\rb^{1-\sigma}\lb\log_2 q\rb^{1-\sigma}.
\end{align*}
\end{conjecture}

 And we have the following conjecture for the maximal value of  the logarithmic derivative at $\frac{1}{2}$.

\begin{conjecture}\label{CJDL3}
There exists constants $A>0$ and $D >\frac{1}{2}$ such that  as prime $q \to \infty$, we have \begin{align*}
\max_{ \substack{  \chi \neq \chi_0 \\ \chi(\emph{mod}\, q)}} \emph{Re} \lb e^{-\i \theta} \frac{L^{\prime}}{L}\lb\frac{1}{2}, \chi\rb\rb
\sim A \sqrt{\log q} \lb\log_2 q \rb^{D},
\end{align*}
uniformly for all $\theta \in [0, 2\pi]$.    
\end{conjecture}

\begin{Remark}
    It may be true that $D = \frac{3}{2}$
in Conjecture \ref{CJDL3}.\end{Remark}

\begin{Remark}
If we restrict to  Dirichlet characters with fixed order, then variants of Conjecture \ref{CJDL1}, \ref{CJDL2}, \ref{CJDL3}  should remain true.   More precisely,  let $k\geqslant 2$ be a fixed integer, let $q$ be prime and let $\J^{(k)}_q$ be the set of Dirichlet characters (mod $q$) with order equal to $k$. We may consider maximal values of Re$\lb e^{-\i \theta}L^{\prime}(\sigma, \chi)/L(\sigma, \chi) \rb$, when  impose the condition $\chi \in \J^{(k)}_q$. If we replace the claim ``uniformly for all $\theta \in [0, 2\pi]$" by the claim ``uniformly for all $\theta \in \{0, \frac{2\pi}{k}, \cdots, \frac{2\pi}{k}\lb k-1\rb\}$'' and allow $q$ varying from $Q$ to $2Q$, then the modified  Conjecture \ref{CJDL1}, \ref{CJDL2}, \ref{CJDL3}  should also be true. 
\end{Remark}

\begin{Remark}
  In the statements of Conjecture \ref{CJDL1}, \ref{CJDL1Z}, \ref{CJDL1Z1}, \ref{CJDL2}, \ref{CJDL3},   if we replace the real part  ``Re$\lb e^{-\i \theta} \cdot \rb$"  of the object by the absolute value``$\left|\,\cdot\,\right|$", then the modified  conjectures should remain true. In particular, this implies that if $-\text{Re} (L^{\prime}(1, \chi)/L(1, \chi)) = \log_2 q + \log_3 q + O(1)$ for some $\chi$, then $\text{Im} (L^{\prime}(1, \chi)/L(1, \chi)) = O(\sqrt{\log_2 q})$. \end{Remark}

In the following, we propose some  problems for further study.
  \begin{problem}
Fix $\theta \in (0,2\pi)$.  Is it true that 
   \begin{align*}
&\max_{T \leqslant t \leqslant  2T}\emph{Re} \lb e^{-\i \theta }\zeta^{\prime}\lb 1 + \i t\rb\rb  \sim A_{\theta}\lb \log_2 T\rb^2, \quad as\quad T \to \infty,\\
&\max_{ \substack{  \chi \neq \chi_0 \\ \chi(\emph{mod}\, q)}} \emph{Re} \lb e^{-\i \theta} L^{\prime}\lb 1, \chi\rb\rb \sim A_{\theta} \lb \log_2 q\rb^2, \quad as\quad prime \quad q \to \infty,
\end{align*}
for some constant $A_{\theta}$?
And what  the correct value $A_{\theta}$  should be conjectured?
\end{problem}

\begin{problem}
Fix $\theta \in (0,2\pi)$.  Is it true that 
   \begin{align*}
&\max_{T \leqslant t \leqslant  2T}\emph{Re} \lb e^{-\i \theta }\zeta\lb 1 + \i t\rb\rb  \sim B_{\theta} \lb \log_2 T\rb, \quad as\quad T \to \infty,\\
&\max_{ \substack{  \chi \neq \chi_0 \\ \chi(\emph{mod}\, q)}} \emph{Re} \lb e^{-\i \theta} L\lb 1, \chi\rb\rb \sim B_{\theta} \lb \log_2 q\rb, \quad as\quad prime \quad q \to \infty,
\end{align*}
for some constant $B_{\theta}$?
And what  the correct value $B_{\theta}$  should be conjectured?
\end{problem}

\begin{problem}
Fix $\theta \in (0,2\pi)$ and $k>0$. Study   the size for the following extreme values 
   \begin{align*}
\max_{T \leqslant t \leqslant  2T}\emph{Re} \lb e^{-\i \theta }\zeta\lb \frac{1}{2} + \i t\rb\rb, \quad \quad \max_{ \substack{  \chi \neq \chi_0 \\ \chi(\emph{mod}\, q)}} \emph{Re} \lb e^{-\i \theta} L\lb \frac{1}{2}, \chi\rb\rb 
\end{align*}
and establish lower bounds for the following moments
\begin{align*}
\int_{T \leqslant t \leqslant  2T} \lb\emph{Re} \lb e^{-\i \theta }\zeta\lb \frac{1}{2} + \i t\rb\rb\rb^{k}\d t\,, \quad \quad \sum_{ \substack{  \chi \neq \chi_0 \\ \chi(\emph{mod}\, q)}} \lb\emph{Re} \lb e^{-\i \theta} L\lb \frac{1}{2}, \chi\rb\rb\rb^{k}. 
\end{align*}
\end{problem}

And one can consider above problems for higher order derivatives.

\section*{Acknowledgements}
 I thank Christoph Aistleitner  for his encouragement. I thank Kannan Soundararajan   for telling me the connection between  the Euler-Kronecker constants and the logarithmic derivatives of Dirichlet $L$-functions, and for drawing my attention to the work of Lamzouri. Half of this work was conducted in the spring of 2021 when I was   supported by the Austrian Science Fund (FWF), project W1230. I am currently supported  by a postdoctoral fellowship funded by  the Courtois Chair II in fundamental research;
 the Natural Sciences and Engineering Research Council of Canada; and
the Fonds de recherche du Qu\'ebec - Nature et technologies.

\bigskip

\end{document}